\RequirePackage{fix-cm}

\documentclass[smallextended]{svjour3}       
\smartqed  

\usepackage{graphicx}
\usepackage{mathptmx}      

\usepackage{hyperref}
\usepackage{amsfonts}
\usepackage{amsmath,booktabs,ctable,threeparttable}
\usepackage{amssymb,amsfonts,boxedminipage}
\usepackage[caption=false]{subfig}
\usepackage{multirow}
\usepackage{color}
\usepackage{enumerate}

\usepackage{amsthm}
\usepackage{algorithm}
\usepackage{algorithmic}
\usepackage{lineno}
\usepackage{marvosym}
\usepackage{appendix}

\newtheorem{Def}{Definition}[section]

\newtheorem{Lem}{Lemma}[section]
\numberwithin{equation}{section}
\numberwithin{theorem}{section}
\numberwithin{Lem}{section}
\numberwithin{Def}{section}
\numberwithin{remark}{section}
\numberwithin{prop}{section}

\begin{document}

\title{The joint bidiagonalization process with partial reorthogonalization
	\thanks{This
		work was supported in part by
		the National Science Foundation of China (No. 11771249).}}


\titlerunning{JBD with partial reorthogonalization}        

\author{Zhongxiao Jia \and Haibo Li }


\institute{ Zhongxiao Jia \at
	Department of Mathematical Sciences, Tsinghua University, 100084 Beijing, China \\
	\email{jiazx@tsinghua.edu.cn}
	\and
	Haibo Li \at
	Department of Mathematical Sciences, Tsinghua
	University, 100084 Beijing, China. \\
	\email{li-hb15@mails.tsinghua.edu.cn}
}


\maketitle
\begin{abstract}
The joint bidiagonalization(JBD) process is a useful algorithm for the computation of the generalized singular value decomposition(GSVD) of a matrix pair. However, it always suffers from rounding errors, which causes the Lanczos vectors to loss their mutual orthogonality. In order to maintain some level of orthongonality, we present a semiorthogonalization strategy. Our rounding error analysis shows that the JBD process with the semiorthogonalization strategy can ensure that the convergence of the computed quantities is not affected by rounding errors and the final accuracy is high enough. Based on the semiorthogonalization strategy, we develop the joint bidiagonalization process with partial reorthogonalization(JBDPRO). In the JBDPRO algorithm, reorthogonalizations occur only when necessary, which saves a big amount of reorthogonalization work compared with the full reorthogonalization strategy. Numerical experiments illustrate our theory and algorithm.
	
\end{abstract}	

\section{Introduction}\label{sec1}
The joint bidiagonalization(JBD) process is a useful algorithm for computing some extreme generalized singular values and vectors for a large sparse or structured matrix pair $\{A,L\}$ \cite{Paige1981,Van1976} where $A\in\mathbb{R}^{m\times n}$ and $L\in\mathbb{R}^{p\times n}$, as well as solving large-scale discrete ill-posed problems with general-form Tikhonov regularization \cite{Hansen1989,Hansen1998,Hansen2010}. First proposed by Zha \cite{Zha1996}, it iteratively reduces the matrix pair $\{A,L\}$ to an upper or lower bidiagonal form. It was later adopted by Kilmer \cite{Kilmer2007} to jointly diagonalizes $\{A, L\}$ to lower and upper bidiagonal forms.

Consider the compact $QR$ factorization of the stacked matrix:
\begin{equation}\label{1.1}
\begin{pmatrix}
A \\
L
\end{pmatrix} = QR =
\begin{pmatrix}
Q_{A} \\
Q_{L}
\end{pmatrix}R  ,
\end{equation}
where $Q \in \mathbb{R}^{(m+p)\times n}$ is column orthonormal and $R\in \mathbb{R}^{n\times n}$. We partition $Q$ such that $Q_{A}\in\mathbb{R}^{m\times n}$ and $Q_{L}\in\mathbb{R}^{p\times n}$, so we have $A= Q_{A}R$ and $L = Q_{L}R$. Applying the BIDIAG-1 procedure and BIDIAG-2 procedure \cite{Paige1982}, which correspond to the lower and upper Lanczos bidiagonalization processes, to $Q_{A}$ and $Q_{L}$, respectively, we can reduce $Q_{A}$ and $Q_{L}$ to the following lower and upper bidiagonal matrices, respectively:
\begin{equation}\label{1.2}
B_{k}=\begin{pmatrix}
\alpha_{1} & & & \\
\beta_{2} &\alpha_{2} & & \\
&\beta_{3} &\ddots & \\
& &\ddots &\alpha_{k} \\
& & &\beta_{k+1}
\end{pmatrix}\in  \mathbb{R}^{(k+1)\times k} , \ \
\widehat{B}_{k}=\begin{pmatrix}
\hat{\alpha}_{1} &\hat{\beta}_{1} & & \\
&\hat{\alpha}_{2} &\ddots & \\
& &\ddots &\hat{\beta}_{k-1} \\
& & &\hat{\alpha}_{k}
\end{pmatrix}\in  \mathbb{R}^{k\times k}  .
\end{equation}
The two processes produce four column orthonormal matrices, that is 
\begin{equation}\label{1.3}
U_{k+1}=(u_{1},\dots,u_{k+1}) \in \mathbb{R}^{m\times (k+1)} , \ \ 
V_{k}=(v_{1},\dots,v_{k}) \in \mathbb{R}^{n\times k}
\end{equation}
computed by the BIDIAG-1 algorithm, and 
\begin{equation}\label{1.4}
\widehat{U}_{k}=(\hat{u}_{1},\dots,\hat{u}_{k}) \in \mathbb{R}^{p\times k}, \ \ 
\widehat{V}_{k}=(\hat{v}_{1},\dots,\hat{v}_{k}) \in \mathbb{R}^{n\times k} 
\end{equation}
computed by the BIDIAG-2 algorithm.

In order to combine BIDIAG-1 and BIDIAG-2, the starting vector of BIDIAG-2 is
chosen to be $\hat{v}_{1}= v_{1}$ and the upper bidiagonalization of $Q_{L}$ continues. It is proved in \cite{Zha1996,Kilmer2007} that the Lanczos vector $\hat{v}_{i}$ and the element $\hat{\beta}_{i}$ of $\widehat{B}_{k}$ can be computed by using the following relations:
\begin{equation}\label{1.5}
\hat{v}_{i+1} = (-1)^{i}v_{i+1} , \ \ 
\hat{\beta}_{i} = \alpha_{i+1}\beta_{i+1}/\hat{\alpha}_{i} .
\end{equation}
For large-scale matrices $A$ and $L$, the explicitly $QR$ factorization \eqref{1.1} can be avoided by solving a least squares problem with $(A^{T}, L^{T})^{T}$ as the coefficient matrix iteratively at each iteration \cite{Bjorck1996,Paige1982}. Through the above modifications, we obtain the JBD process which can efficiently reduce a large-scale matrix pair $\{A,L\}$
to a bidiagonal matrix pair $\{B_{k}, \widehat{B}_{k}\}$. For details of the derivation of the algorithm, see \cite{Zha1996,Kilmer2007}. In exact arithmetic, the $k$-step JBD process explicitly computes three column orthonormal matrices $U_{k+1}$, $\widetilde{V}_{k}$, $\widehat{U}_{k}$, a lower bidiagonal matrix $B_{k}$ and an upper bidiagonal matrix $\widehat{B}_{k}$. The two column orthonormal matrices $V_{k}$ and $\widehat{V}_{k}$ can be obtained from $\widetilde{V}_{k}$ implicitly by letting $V_{k}=Q^{T}\widehat{V}_{k}$ and $\widehat{V}_{k}=V_{k}P$, where $ P=diag(1,-1,\dots ,(-1)^{k-1})$. 

The JBD process can be used to approximate a few largest or smallest generalized singular values and corresponding vectors of $\{A,L\}$ by projecting the original large-scale problem to the reduced small-scale problem $\{B_{k}, \widehat{B}_{k}\}$. Furthermore, Kilmer \textit{et al.} \cite{Kilmer2007} present an iterative method based on the JBD process to solve ill-posed problems with general-form Tikhonov regularization. The main idea is to use the projection method to solve a series of small-scale general-form Tikhonov regularization problems which lies in lower dimensional subspaces. Jia and Yang \cite{JiaYang2018} have analyzed this iterative regularized method and they present a new iterative regularized algorithm.

In exact arithmetic, the $k$-step JBD algorithm is equivalent to the combination of the lower and upper Lanczos bidiagonalization processes. The lower Lanczos bidiagonalization process computes two column orthonormal matrices $U_{k+1}$ and $V_{k}$, while the upper Lanczos bidiagonalization process computes two column orthonormal matrices $\widehat{U}_{k}$ and $\widehat{V}_{k}$. In finite precision arithmetic, however, the orthogonality of Lanczos vectors computed by the JBD process is gradually lost, which is due to the influence of rounding errors. For the GSVD computation, the loss of orthogonality of Lanczos vectors will lead to a delay of the convergence of Ritz values and it causes the appearance of spurious generalized singular values, which are called ``ghosts" \cite{Zha1996,Li2019}. In order to preserve the convergence of the approximate generalized singular values, we need to perform the JBD process with a reorthogonalization strategy to maintain some level of orthogonality of the Lanczos vectors.

The loss of orthogonality of Lanczos vectors is a typical phenomenon appeared in the Lanczos-type algorithms, which is first observed in the symmetric Lanczos process \cite{Lanczos1950}. It will lead to a delay of convergence in the computation of some extreme eigenvalues of a symmetric matrix \cite{Paige1971,Paige1972,Paige1980,Meurant2006}, and sometimes it is also difficult to determine whether some computed approximations are additional copies or genuine close eigenvalues \cite{Paige1971,Paige1972,Paige1976,Paige1980}. In order to preserve the convergence, a few reorthogonalization strategies have been proposed to maintain some level of orthogonality \cite{Parlett1979,Simon1984a,Simon1984b,Parlett1992}. Especially, Simon \cite{Simon1984b} proves that semiorthogonality of Lanczos vectors is enough to guarantee the accuracy of the computed quantities and avoid spurious eigenvalues from appearing. The above results of the symmetric Lanczos process have been adapted by Larsen to handle the Lanczos bidiagonalization process, and he proposes the Lanczos bidiagonalization with partial reorthogonalization algorithm \cite{Larsen1998}, which can save a big amount of reorthogoanlization work compared with the full reorthogonalization strategy. In \cite{SimonZha2000}, Simon and Zha propose a one-sided reorthogonalization strategy for the Lanczos bidiagonalization process. Later in \cite{Barlow2013}, the Lanczos bidiagonalization process with the one-sided reorthogonalization has been analyzed in detail by Barlow.

In this paper, we propose a semiorthogonalization strategy for the $k$-step JBD process to keep the orthogonality levels of $u_{i}$, $\tilde{v}_{i}$ and $\hat{u}_{i}$ below $\sqrt{\epsilon/(2k+1)}$, where $\epsilon$ is the roundoff unit. We make a rounding error analysis of the JBD process with the semiorthogonalization strategy, which establishes connections between the JBD process with the semiorthogonalization strategy and the Lanczos bidiagonalization process in finite precision arithmetic. The approximate generalized singular values of $\{A,L\}$ can be computed by using the singular value decomposition (SVD) of either $B_{k}$ or $\widehat{B}_{k}$ \cite{Zha1996,Li2019}. We will prove that semiorthogonality of the Lanczos vectors are enough to preserve convergence of the Ritz values computed from either $B_{k}$ or $\widehat{B}_{k}$, and the generalized singular values can be approximated with high accuracy by using the SVD
of $B_k$, while the accuracy of the approximated generalized singular values computed from $\widehat{B}_{k}$ is high enough as long as $\|\widehat{B}_{k}^{-1}\|$ does not become too large.

Based on the semiorthogonalization strategy, we develop a practical algorithm called the joint bidiagonalization process with partial reorthogonalization(JBDPRO). The central idea in partial reorthogonalization is that the levels of orthogonality among the Lanczos vectors satisfy a coupled recurrence relations \cite{Simon1984b,Larsen1998}, which can be used as a practical tool for computing estimates of the levels of orthogonality in an efficient way and to decide when to reorthogonalize, and which Lanczos vectors are necessary to include in the reorthogonalization step. Numerical experiments shows that our JBDPRO algorithm is more efficient than the joint bidiagonalization process with full reorthognalization(JBDFRO), while can avoid ``ghosts" from appearing.

This paper is organized as follows. In Section \ref{sec2}, we review the JBD process with some properties, and we review the GSVD computation based on the JBD process. In Section \ref{sec3}, we propose a semiorthogonalization strategy, and make a detailed analysis of the JBD process with the semiorthogonalization strategy. Based on the semiorthogonalization strategy, in Section \ref{sec4}, we develop the JBDPRO algorithm. In Section \ref{sec5}, we use some numerical examples to illustrate our theory and algorithm. Finally, we conclude the paper in Section \ref{sec6}.

Throughout the paper, we denote by $I_k$ the identity matrix of order $k$, by $0_k$ and $0_{k\times l}$ the zero matrices of order $k$ and $k\times l$, respectively. The subscripts are omitted when there is no confusion. We denote by $span(C)$ the subspace spanned by columns of a matrix $C$. The transpose of a matrix $C$ is denoted by $C^{T}$. The roundoff unit is denoted by $\epsilon$. The norm $\|\cdot \|$ always means the spectral or 2-norm of a matrix or vector.

\section{Joint bidiagonalization process and GSVD computation}\label{sec2}
In this section, we review the joint bidiagonalization process and its basic properties in both exact and finite precision arithmetic. We also describe the GSVD computation of $\{A,L\}$ based on the JBD process.

The joint bidiagonalization process is described in Algorithm \ref{alg1}. Notice that for large-scale matrices $A$ and $L$, the explicitly $QR$ factorization \eqref{1.1} is impractical due to efficiency and storage. At each iteration $i=1,2,\ldots,k+1$, Algorithm~\ref{alg1} needs to compute $QQ^{T} \begin{pmatrix}
u_i \\ 0_p 
\end{pmatrix}$, which is not accessible since $Q$ is not available. Let $\tilde{u}_i=\begin{pmatrix}
u_i \\ 0_p 
\end{pmatrix}$. Notice that $QQ^T\tilde{u}_i$ is nothing but the orthogonal projection
of $\tilde{u}_i$ onto the column space of 
$\begin{pmatrix} A \\ L
\end{pmatrix}$, which means that $QQ^T\tilde{u}_i=\begin{pmatrix}
A \\ L 
\end{pmatrix}\tilde{x}_i$, where
\begin{equation}\label{2.1}
\tilde{x}_i=\arg\min_{\tilde{x}\in \mathbb{R}^n}
\left\|\begin{pmatrix}
A \\ L 
\end{pmatrix}
\tilde{x}-\tilde{u}_i\right\|.
\end{equation}
The large-scale least squares problem \eqref{2.1} can be solved by an iterative solver, e.g., the most commonly used LSQR algorithm \cite{Paige1982}.
\begin{algorithm}[htb]
	\caption{The $k$-step joint bidiagonalization(JBD) process }
	\begin{algorithmic}[1]\label{alg1}
		\STATE {Choosing a starting vector $b \in \mathbb{R}^{m}$,
			$\beta_{1}u_{1}=b,\ \beta_{1}=\| b\|$ }
		\STATE {$\alpha_{1}\tilde{v}_{1}=QQ^{T}\begin{pmatrix}
			u_{1} \\
			0_{p}
			\end{pmatrix} $}
		\STATE {  $\hat{\alpha}_{1}\hat{u}_{1}=\tilde{v}_{1}(m+1:m+p) $}
		\FOR{$i=1,2,\ldots,k,$}
		\STATE $\beta_{i+1}u_{i+1}=\tilde{v}_{i}(1:m)-\alpha_{i}u_{i} $
		\STATE $ \alpha_{i+1}\tilde{v}_{i+1}=
		QQ^{T}\begin{pmatrix}
		u_{i+1} \\
		0_{p}
		\end{pmatrix}-\beta_{i+1}\tilde{v}_{i} $
		\STATE $\hat{\beta}_{i}=(\alpha_{i+1}\beta_{i+1})/\hat{\alpha}_{i} $
		\STATE $\hat{\alpha}_{i+1}\hat{u}_{i+1}=
		(-1)^{i}\tilde{v}_{i+1}(m+1:m+p)-\hat{\beta}_{i}\hat{u}_{i} $
		\ENDFOR
	\end{algorithmic}
\end{algorithm}

Algorithm \ref{alg1} is actually an approach to jointly bidiagonalized $Q_{A}$ and $Q_L$ as a prelude to the $CS$ decomposition of $\{Q_{A},Q_{L}\}$, where the computation of $QR$ factorization \eqref{1.1} is avoided and all we need is an approximation to the orthogonal projection $QQ^{T}$, which can be accessed by solving \eqref{2.1} iteratively. In exact arithmetic, the $k$-step JBD process produces two bidiagonal matrices $B_{k}$, $\widehat{B}_{k}$ and three column orthonormal matrices $U_{k+1}$, $\widehat{U}_{k}$ and
\begin{align}\label{2.2}
\widetilde{V}_{k}=(\tilde{v}_{1},\dots,\tilde{v}_{k}) \in \mathbb{R}^{(m+p)\times k} 
\end{align}
satisfying $\tilde{v}_{i}=Qv_{i}$. We have $v_{i}=Q^{T}\tilde{v}_{i}$ and $\hat{v}_{i}=(-1)^{i-1}v_{i}$, which can be obtained implicitly from $\tilde{v}_{i}$. The first $k$ steps of the recurrences from Algorithm \ref{alg1} are captured in matrix form as
\begin{align}
& (I_{m},0_{m\times p})\widetilde{V}_{k}=U_{k+1}B_{k} \label{2.3} , \\
& QQ^{T}
\begin{pmatrix}
U_{k+1} \\
0_{p\times (k+1)}
\end{pmatrix}
=\widetilde{V}_{k}B_{k}^{T}+\alpha_{k+1}\tilde{v}_{k+1}e_{k+1}^{T} \label{2.4}  ,  \\
& (0_{p\times m},I_{p})\widetilde{V}_{k}P=\widehat{U}_{k}\widehat{B}_{k} \label{2.5} ,
\end{align}
where $ P=diag(1,-1,1,\dots ,(-1)^{k-1})$, and $e_{k+1}$ is the $(k+1)$-th column of the identity matrix of order $k+1$.
In exact arithmetic, one can verify that 
\begin{align}
&Q_AV_k=U_{k+1}B_k,\ \ Q_A^TU_{k+1}=V_kB_k^T+\alpha_{k+1}v_{k+1}e_{k+1}^T,\label{2.6}\\
&Q_L\widehat{V}_k=\widehat{U}_k\widehat{B}_k,\ \ Q_L^T\widehat{U}_k
=\widehat{V}_k\widehat{B}_k^T+\hat{\beta}_k\hat{v}_{k+1}e_k^T , \label{2.7}
\end{align}
where $e_{k}$ the $k$-th column of the identity matrix of order $k$. Therefore, the JBD process of $\{A, L\}$ is equivalent to the combination of the lower and upper Lanczos bidiagonalizations of $Q_{A}$ and $Q_{L}$.  

The JBD process can be used to approximate some extreme generalized singular values and vectors of a large sparse or structured matrix pair $\{A,L\}$. We first describe the GSVD of $\{A,L\}$. Let 
\begin{equation}\label{2.8}
Q_{A} = P_{A}C_{A}W^{T}  , \ \  Q_{L} = P_{L}S_{L}W^{T}
\end{equation}
be the $CS$ decomposition of the matrix pair $\{Q_{A}, Q_{L} \}$ \cite{Van1985}, where $P_{A}=(p_{1}^{A}, \dots,p_{m}^{A})\in \mathbb{R}^{m\times m}$, $P_{L}=(p_{1}^{L}, \dots,p_{p}^{L})\in \mathbb{R}^{p\times p}$ and $W=(w_{1},\dots, w_{n})\in\mathbb{R}^{n\times n}$
are orthogonal matrices, and $C_{A}\in\mathbb{R}^{m\times n}$ and $S_{L}\in\mathbb{R}^{p\times n}$ are diagonal matrices(not necessarily square) satisfying
$C_{A}^{T}C_{A}+S_{L}^{T}S_{L}=I_{n}$. If we add the assumption that $(A^{T}, L^{T})^{T}$ has full column rank, the GSVD of $\{A, L\}$ is
\begin{equation}\label{2.9}
A = P_{A}C_{A}G^{-1}  , \ \  L = P_{L}S_{L}G^{-1}
\end{equation}
with $G=R^{-1}W\in\mathbb{R}^{n\times n}$. The $i$-th generalized singular value of $\{A,L\}$ is $c_{i}/s_{i}$, while the $i$-th corresponding generalized singular vectors are $g_{i}=R^{-1}w_{i}$, $p^{A}_{i}$ and $p_{i}^{L}$. We call $g_{i}$ the $i$-th right generalized singular vector, $p^{A}_{i}$ and $p_{i}^{L}$ the $i$-th left generalized singular vectors corresponding to $A$ and $L$, respectively. Since $c_{i}/s_{i}=\infty$ when $s_{i}=0$, we use the number pair $\{c_{i}, s_{i}\}$ to denote $c_{i}/s_{i}$.
   
After the $k$-step JBD process of $\{A, L\}$, we have computed $B_{k}$ and $\widehat{B}_{k}$. Let us assume that we have computed the compact SVD of $B_{k}$:
\begin{equation}\label{2.10}
	B_{k} = P_{k}\Theta_{k}W_{k}^{T}, \ \ \Theta_{k}=diag(c_{1}^{(k)}, \dots, c_{k}^{(k)}), \ \ 
	1 \geq c_{1}^{(k)} > \dots > c_{k}^{(k)} \geq 0 \ ,
\end{equation}
where $P_{k}=(p_{1}^{(k)}, \dots, p_{k}^{(k)})\in\mathbb{R}^{(k+1)\times k}$ and $W_{k}=(w_{1}^{(k)}, \dots, w_{k}^{(k)})\in\mathbb{R}^{k\times k}$ are column orthonormal, and $\Theta_{k}\in\mathbb{R}^{k\times k}$. The decomposition \eqref{2.10} can be achieved by a
variety of methods since $B_k$ is a bidiagonal matrix of relatively small
dimension. The approximate generalized singular value of $\{A, L\}$ is $\{c_{i}^{(k)}, (1-(c_{i}^{(k)})^{2})^{1/2}\}$, while the approximate right vector is $x_{i}^{(k)}=R^{-1}V_{k}w_{i}^{(k)}$ and the approximate left vector corresponding to $A$ is $y_{i}^{(k)} = U_{k+1}p_{i}^{(k)}$. 
For large-scale matrices $A$ and $L$, the explicit computation of $R^{-1}$ can be avoided. Notice that
$$\begin{pmatrix}
A \\
L
\end{pmatrix}x_{i}^{(k)}=QRR^{-1}V_{k}w_{i}^{(k)}=\widetilde{V}_{k}w_{i}^{(k)} .$$
Hence by solving a least squares problem, we can obtain $x_{i}^{(k)}$ from $\widetilde{V}_{k}w_{i}^{(k)}$.
If we also want to compute the approximate left generalized singular vectors corresponding to $L$, we need to compute the SVD of $\widehat{B}_{k}$. The approximate generalized singular values and corresponding right vectors can also be computed from the SVD of $\widehat{B}_{k}$. The procedure is similar as the above and we omit it; for details see \cite{Zha1996,Li2019}. 

The above method for computing the GSVD of $\{A,L\}$ is an indirect procedure to compute the $CS$ decomposition \eqref{2.8} of $\{Q_{A},Q_{L}\}$, where the computation of $QR$ factors $Q$ and $R$ is avoided and all we need is an approximation to the orthogonal projection $QQ^{T}$, which can be accessed by solving \eqref{2.1} iteratively.

In finite precision arithmetic, by the influence of rounding errors, the behavior of the JBD process will deviate far from the ideal case in exact arithmetic, and the convergence and accuracy of the approximate generalized singular values and vectors computed by using the JBD process will be affected. The rounding error analysis of the JBD process in finite precision arithmetic is based on a set of assumptions and properties of the behavior of the rounding errors occurring, which constitutes a rational model for the actual computation. We state them here following \cite{Li2019}.  

First, we always assume that \eqref{2.1} is solved accurately at each iteration. Thus the computed 
$\begin{pmatrix}
A \\ L 
\end{pmatrix}\tilde{x}_i$
is equal to the value of
$QQ^{T} \begin{pmatrix}
u_i \\ 0_p 
\end{pmatrix}$
computed by explicitly using the strictly column orthonormal matrix $Q$. Second, the rounding errors appeared in the computation at each step are assumed to be of order $O(\epsilon)$. Third, the property of local orthogonality of $u_{i}$ and $\hat{u}_{i}$ holds, that is, locally the orthogonality levels of $u_{i}$ and $\hat{u}_{i}$ satisfy the following relations respectively:
\begin{align}
& \beta_{i+1}|u_{i+1}^{T}u_{i}| = O(c_{1}(m,n)\epsilon) , \label{2.11} \\
& \hat{\alpha}_{i+1}|\hat{u}_{i+1}^{T}\hat{u}_{i}| = O(c_{2}(p,n)\epsilon) ,\label{2.12} 
\end{align}
where $c_{1}(m,n)$ and $c_{2}(p,n)$ are two modestly growing functions of $m$, $n$ and $p$. Finally, we assume that 
\begin{equation}\label{2.13}
no \ \ \alpha_{i}, \ \beta_{i+1}, \ \hat{\alpha}_{i} \ and \  \hat{\beta}_{i} \  ever \  become \  negligible ,
\end{equation}
which is almost always true in practice, and the rare cases where $\alpha_{i}$, $\beta_{i+1}$, $\hat{\alpha}_{i}$ or $\hat{\beta}_{i}$ do become small are actually the lucky ones, since then the algorithm should be terminated, having found an invariant singular subspace. Besides, we always assume that the computed Lanczos vectors are of unit length. 

Under the above assumptions, it has been shown in \cite{Li2019} that 
\begin{equation}\label{2.14}
\lVert \widetilde{V}_{k} - QV_{k} \lVert 
= O(\lVert \underline{B}_{k}^{-1}\lVert\epsilon) 
\end{equation}
with
$\underline{B}_{k}=\begin{pmatrix}
B_{k-1}^{T} \\
\alpha_{k}e_{k}^{T}
\end{pmatrix}\in \mathbb{R}^{k\times k}$, which implies that $\widetilde{V}_{k}$ gradually deviates from the column space of $Q$ as the iterations progress. Furthermore, the following four relations hold:
\begin{align}
& Q_{A}V_{k}=U_{k+1}B_{k}+F_{k} , \ \ 
Q_{A}^{T}U_{k+1}=V_{k}B_{k}^{T}+\alpha_{k+1}v_{k+1}e_{k+1}^{T}+G_{k+1} \label{2.15}  ,  \\
& Q_{L}\widehat{V}_{k}=\widehat{U}_{k}\widehat{B}_{k}+\widehat{F}_{k}, \ \ \ \ \ \ 
Q_{L}^{T}\widehat{U}_{k}=\widehat{V}_{k}\widehat{B}_{k}^{T}\ +\hat{\beta}_{k}\hat{v}_{k+1}e_{k}^{T} + \widehat{G}_{k} \label{2.16}  ,
\end{align}
where 
\begin{align}
& \| F_{k} \| = O(\|\underline{B}_{k}^{-1}\|\epsilon), \label{2.17} \ \ 
\|G_{k+1}\|= O(\epsilon), \\
& \|\widehat{F}_{k}\| = O(\|\underline{B}_{k}^{-1}\|\epsilon), \ \  \|\widehat{G}_{k}\|=O((\|\underline{B}_{k}^{-1}\|+\|\widehat{B}_{k}^{-1}\|)\epsilon) \label{2.18} .
\end{align}

\begin{remark}\label{rem2.1}
The growth speed of $\|\underline{B}_{k}^{-1}\|$ can be controlled. In the GSVD computation problems, usually at least one matrix of $\{A, L\}$ is well conditioned, which results to that at least one of $\{Q_{A}, Q_{L}\}$ is well conditioned. If $Q_{A}$ is the well conditioned one, we implement the JBD process of $\{A, L\}$, while if $Q_{L}$ is the well conditioned one, we implement the JBD process of $\{L, A\}$.  By this modification, we could always make sure that $\underline{B}_{k}$ is a well conditioned matrix and $\|\underline{B}_{k}^{-1}\|$ does not become too large. 
\end{remark}

Following Remark \ref{rem2.1}, we can always assume that $\|\underline{B}_{k}^{-1}\|=O(1)$. Thus we can make sure that $\tilde{v}_{i}$ is approximately in the subspace spanned by the columns of $Q$ within error $O(\epsilon)$, and $\| F_{k} \|, \ \|\widehat{F}_{k}\|$ are about $O(\epsilon)$. Therefore, \eqref{2.15} indicates that the process of computing $U_{k+1}$, $V_{k}$ and $B_{k}$ can be treated as the lower Lanczos bidiagonalization of $Q_{A}$ within error $O(\epsilon)$. However, if $\|\widehat{B}_{k}^{-1}\|$ becomes too large, the process of computing $\widehat{U}_{k}$, $\widehat{V}_{k}$ and $\widehat{B}_{k}$ will deviate far from the upper Lanczos bidiagonalization of $Q_{L}$.

In finite precision arithmetic, the Lanczos vectors computed by the JBD process gradually lose their mutual orthogonality as the iteration number $k$ increases. Following \cite{Li2019}, we give the definition of the orthogonality level of a group of vectors.

\begin{Def}\label{def2.1}
For a matrix $W_{k}=(w_{1}, \dots, w_{k})\in \mathbb{R}^{r\times k}$ with $\lVert w_{j}\lVert=1$, $j=1,\dots,k$, we give two measures of the orthogonality level of $\{w_{1}, \dots, w_{k}\}$ or $W_{k}$:
$$\kappa(W_{k})=\max_{1\leq i\neq j \leq k}|w_{i}^{T}w_{j}| , \ \ \ \ 
\xi(W_{k}) = \| I_{k}-W_{k}^{T}W_{k}\| . $$
\end{Def}

In the following analysis, we often use terminology ``the orthogonality level of $w_{i}$'' for simplicity, which means the orthogonality level of $\{w_{1}, \dots, w_{k}\}$. Notice that $\kappa(W_{k})\leq\xi(W_{k})\leq k\kappa(W_{k})$. In most occasions, the two quantities can be used interchangeably to measure the orthogonality level of Lanczos vectors. We call $w_{i}$ ``semiorthogonal" if its orthogonality level is about $\sqrt{\epsilon}$. Using the method appeared in \cite{Barlow2013}, we can obtain $\|W_{k}\| \leq \sqrt{1 + \xi(W_{k})}$. This upper bound will be used later.


If we use the JBD process to approximate some generalized singular values and vectors of $\{A,L\}$, the loss of orthogonality of Lanczos vectors will lead to a delay of the convergence of Ritz values and the appearance of ``ghosts". To preserve the convergence, one can use the full reorthogonalization for $u_{i},\hat{u}_{i}$ and $\tilde{v}_{i}$ at each iteration, to make sure that the orthogonality levels of $u_{i},\hat{u}_{i}$ and $\tilde{v}_{i}$ are about $O(\epsilon)$. The disadvantage of full reorthogonalization strategy is that it will cause too much extra computation. It has been shown in \cite{Li2019} that semiorthogonality of Lanczos vectors are enough to guarantee the accuracy of the approximate generalized singular values and avoid ``ghosts" from appearing. In the next section, we will propose a semiorthogonalization strategy, and make a detailed analysis of the JBD process equipped with the semiorthogonalization strategy.

\section{A semiorthogonalization strategy for the JBD process}\label{sec3}
Now we introduce a semiorthogonalization strategy for the JBD process. The semiorthogonalization strategy is similar to that proposed by Simon for the symmetric Lanczos process \cite{Simon1984b}. We use the reorthogonalization of $u_{i+1}$ to describe it. Let $\omega_{0}=\sqrt{\epsilon/(2k+1)}$. 
At the $i$-th step, suppose that
$$\beta_{i+1}^{'}u_{i+1}^{'}=\tilde{v}_{i}(1:m)-\alpha_{i}u_{i}-f_{i}^{'} .$$
If $|u_{i+1}^{'T}u_{j}| > \omega_{0}$ for some $j < i$, then we choose $i-1$ real numbers $\xi_{1i}, \dots, \xi_{i-1,i}$, and form 
$$\beta_{i+1}u_{i+1} =\beta_{i+1}^{'}u_{i+1}^{'}- 
	\sum\limits_{j=1}^{i-1}\xi_{ji}u_{j}-f_{i}^{''} .$$
In the above equations, $f_{i}^{'}$ and $f_{i}^{''}$ are rounding error terms appeared in the computation. The algorithm will be continued with $u_{i+1}$ instead of $u_{i+1}^{'}$. 

\begin{Def}\label{Def2}
	The above modification of the JBD process will be called a semiorthogonalization stategy for $u_{i+1}$ if the following conditions are satisfied: \\
	(1) The numbers $\xi_{1i}, \dots, \xi_{i-1,i}$ are chosen such that
	\begin{equation}\label{3.1}
		u_{i+1}^{T}u_{j} \leq \omega_{0} \ , \ j = 1, \dots, i .
	\end{equation}
	\\
	(2) The computation of $u_{i+1}$ can be written as
	\begin{equation}\label{3.2}
	\beta_{i+1}u_{i+1} = \tilde{v}_{i}(1:m)-\alpha_{i}u_{i} - 
	\sum\limits_{j=1}^{i-1}\xi_{ji}u_{j}-f_{i}  ,
	\end{equation}
	where $f_{i} = f_{i}^{'}+f_{i}^{''}$ is the rounding error term, satisfying $\|f_{i}\| = O(q_{1}(m,n)\epsilon)$ with $q_{1}(m,n)$ a modestly growing function of $m$ and $n$.
\end{Def}

The semiorthogonalization stategy for $\tilde{v}_{i+1}$  and $\hat{u}_{i+1}$ are similar, and the corresponding $i$-th step recurrences are
\begin{align}
	& \alpha_{i+1}\tilde{v}_{i+1}=
	QQ^{T}\begin{pmatrix}
		u_{i+1} \\
		0_{p}
	\end{pmatrix}-\beta_{i+1}\tilde{v}_{i} - \sum\limits_{j=1}^{i-1}\eta_{ji+1}\tilde{v}_{j} -
	g_{i+1} , \label{3.3} \\
	& \hat{\alpha}_{i+1}\hat{u}_{i+1}=
	(-1)^{i}\tilde{v}_{i+1}(m+1:m+p)-\hat{\beta}_{i}\hat{u}_{i} - \sum\limits_{j=1}^{i-1}\hat{\xi}_{ji+1}\hat{u}_{j}-\hat{f}_{i+1} , \label{3.4}
\end{align}
where $\|g_{i+1}\|=O(q_{2}(m,p)\epsilon)$ and $\|\hat{f}_{i+1}\|=O(q_{3}(p,n)\epsilon)$ with $q_{2}(m,p)$ and $q_{3}(p,n)$ two modestly growing functions of $m$, $n$ and $p$.

Notice that the reorthogonalization of $u_{i+1}$ does not use the vector $u_{i}$, due to the property of local orthogonality among $u_{i}$ and $u_{i+1}$. The reasons are similar for the reorthogonalizations of $\tilde{v}_{i+1}$ and $\hat{u}_{i+1}$. After the semiorthogonalization step, relations \eqref{2.11} and \eqref{2.12} will still hold.

After $k$ steps, we have computed three groups of Lanczos vectors $\{u_{1}, \dots, u_{k+1}\}$,
$\{\tilde{v}_{1}, \dots, \tilde{v}_{k}\}$ and $\{\hat{u}_{1}, \dots, \hat{u}_{k}\}$, of which orthogonality levels are below $\omega_{0}$. The first $k$ steps of the recurrences are captured in matrix form as
\begin{align}
	& (I_{m},0_{m\times p})\widetilde{V}_{k}=U_{k+1}(B_{k}+C_{k})+F_{k} \label{3.5}  ,  \\
	& QQ^{T}
	\begin{pmatrix}
		U_{k+1} \\
		0_{p\times (k+1)}
	\end{pmatrix}
	=\widetilde{V}_{k}(B_{k}^{T}+D_{k})+\alpha_{k+1}\tilde{v}_{k+1}e_{k+1}^{T}+G_{k+1} \label{3.6} , \\
	& (0_{p\times m},I_{p})\widetilde{V}_{k}P=\widehat{U}_{k}(\widehat{B}_{k}+\widehat{C}_{k})+\widehat{F}_{k} \label{3.7} ,
\end{align}
where $\widetilde{F}_{k}=(f_{1}, \dots, f_{k}), \ \widetilde{G}_{k+1} = (g_{1}, \dots, g_{k+1}), \ \bar{F}_{k}=(\hat{f}_{1}, \dots, \hat{f}_{k})$, and 
\begin{align*}
		C_{k}=\begin{pmatrix}
			0 &\xi_{12} &\dots &\xi_{1k} \\
			0& 0&\cdots &\xi_{2k} \\
			& 0&\ddots &\vdots \\
			& &\ddots  &0 \\
			& &&0
		\end{pmatrix} \in \mathbb{R}^{(k+1)\times k}  , \ \
		\widehat{C}_{k} = \begin{pmatrix}
			0  &0 &\hat{\xi}_{13} &\cdots &\hat{\xi}_{1k} \\
			&0 & 0 &\cdots &\hat{\xi}_{2k} \\
			& & \ddots &\ddots &\vdots \\
			& & & 0 &0 \\
			& & & & 0
		\end{pmatrix} \in \mathbb{R}^{k\times k}  , \\
		D_{k}=\begin{pmatrix}
			0 &0 &\eta_{13} &\cdots &\eta_{1k} &\eta_{1k+1} \\
			&0 &0 &\eta_{24} &\cdots &\eta_{2k+1} \\
			& &\ddots &\ddots &\ddots &\vdots \\
			& & &\ddots &0 &\eta_{k-1,k+1} \\
			& & & &0 & 0
		\end{pmatrix} \in \mathbb{R}^{k \times (k+1)} .
\end{align*}

Notice that $\tilde{v}_{i}$ is approximately in the subspace spanned by the columns of $Q$ within error $O(\epsilon)$ for $i = 1, \dots, k$. If we let $v_{i}=Q^{T}\tilde{v}_{i}$ and $\hat{v}_{i}=(-1)^{i-1}v_{i}$, then $\widetilde{V}_{k}=QV_{k}+O(\epsilon)$, and $\{v_{1}, \dots, v_{k}\}$ and $\{\hat{v}_{1}, \dots, \hat{v}_{k}\}$ are also kept semiorthogonal. From \eqref{3.5}--\eqref{3.7} we can obtain
\begin{align}
	& Q_{A}V_{k}=U_{k+1}(B_{k}+C_{k})+F_{k} , \label{3.8} \\
	& Q_{A}^{T}U_{k+1}=V_{k}(B_{k}^{T}+D_{k})+
	\alpha_{k+1}v_{k+1}e_{k+1}^{T}+G_{k+1} \label{3.9} , \\
	& Q_{L}\widehat{V}_{k}=\widehat{U}_{k}(\widehat{B}_{k}+\widehat{C}_{k})+\widehat{F}_{k} , \label{3.10}
\end{align}
where $\|F_{k}\|=O(q_{1}(m,n)\epsilon)$, $\|G_{k+1}\|=O(q_{2}(m,p)\epsilon)$ and $\|\widehat{F}_{k}\|=O(q_{3}(p,n)\epsilon)$. We point out that the rounding error terms $F_{k}$, $G_{k+1}$ and $\widehat{F}_{k}$ here are different from that appeared in relations \eqref{2.15}--\eqref{2.18}, and we use the same notations just for simplicity. 

The following two lemmas describe some basic properties of the JBD process with the semiorthogonalization strategy. The proofs are given in the Appendix \ref{Apd}.
\begin{Lem}\label{Lem3.1}
	For the JBD process with the semiorthogonalization strategy, the relation
	\begin{equation}\label{3.11}
	Q_{L}^{T}\hat{u}_{i} \in span\{\hat{v}_{1}, \dots, \hat{v}_{i+1}\} + O(\bar{q}(m,n,p)\epsilon) .
	\end{equation}
	holds for all $i = 1, 2, \dots $, where $\bar{q}(m,n,p)=q_{1}(m,n)+q_{2}(m,p)+q_{3}(p,n)$.
\end{Lem} 

\begin{Lem}\label{Lem3.2}
	For the $k$-step JBD process with the semiorthogonalization strategy, we have
	\begin{equation}\label{3.12}
	C_{k} = O(\sqrt{\epsilon}) , \ \ D_{k} = O(\sqrt{\epsilon}) , 
	\ \ \widehat{C}_{k} = O(\sqrt{\epsilon}) ,
	\end{equation}
	where $X=O(\sqrt{\epsilon})$ for a matrix $X$ means that all the elements of $X$ are of $O(\sqrt{\epsilon})$.
\end{Lem}

Now we give the relation between the two computed quantities $B_{k}$ and $\widehat{B}_{k}$.
\begin{theorem}\label{Th3.1}
	Given the $k$-step JBD process with the semiorthogonalization strategy, we have
	\begin{equation}\label{3.13}
	B_{k}^{T}B_{k}+P\widehat{B}_{k}^{T}\widehat{B}_{k}P=I_{k} + H_{k} ,
	\end{equation}
	where $H_{k}$ is a symmetric tridiagonal matrix with bandwidth $1$, and the nonzero elements of $H_{k}$ are of $O(c_{3}(m,n,p)\epsilon)$ with $c_{3}(m,n,p)=c_{1}(m,n)+c_{2}(p,n)+q_{1}(m,n)+q_{3}(p,n)$.
\end{theorem}
\begin{proof}
	Since
	\begin{equation*}
	B_{k}^{T}B_{k}=\begin{pmatrix}
	\alpha_{1}^{2}+\beta_{2}^{2} &\alpha_{2}\beta_{2} & & \\
	\alpha_{2}\beta_{2} &\alpha_{2}^{2}+\beta_{3}^{2} &\ddots & \\
	&\ddots &\ddots & \alpha_{k}\beta_{k}\\
	& &\alpha_{k}\beta_{k} &\alpha_{k}^{2}+\beta_{k+1}^{2}
	\end{pmatrix},
	\end{equation*}
	\begin{equation*}
	\widehat{B}_{k}^{T}\widehat{B}_{k}=\begin{pmatrix}
	\hat{\alpha}_{1}^{2}&\hat{\alpha}_{1}\hat{\beta}_{1} & & \\
	\hat{\alpha}_{1}\hat{\beta}_{1}&\hat{\alpha}_{2}^{2}+\hat{\beta}_{1}^{2} &\ddots & \\
	&\ddots &\ddots &\hat{\alpha}_{k-1}\hat{\beta}_{k-1} \\
	& &\hat{\alpha}_{k-1}\hat{\beta}_{k-1} &\hat{\alpha}_{k}^{2}+\hat{\beta}_{k-1}^{2}
	\end{pmatrix},
	\end{equation*}
	nonzero elements in $H_{k}$ are contained only in
	the diagonal and subdiagonal parts.
	
	For the diagonal part, from \eqref{3.2} we have
	\begin{align*}
	\|\tilde{v}_{i}(1:m) \|^{2} 
	& = \| \alpha_{i}u_{i} + \beta_{i+1}u_{i+1}
	+ \sum\limits_{j=1}^{i-1}\xi_{ji}u_{j}+f_{i}\|^{2} \\
	& = \alpha_{i}^{2}+\beta_{i+1}^{2}+2\alpha_{i}\beta_{i+1}u_{i}^{T}u_{i+1} +
	2\alpha_{i}u_{i}^{T}f_{i}+2\beta_{i+1}u_{i+1}^{T}f_{i}+\|f_{i}\|^{2}+ \\
	& \ \  \ \ \| \sum\limits_{j=1}^{i-1}\xi_{ji}u_{j}\|^{2}
	+ 2\alpha_{i}\sum\limits_{j=1}^{i-1}\xi_{ji}u_{i}^{T}u_{j} +
	2\beta_{i+1}\sum\limits_{j=1}^{i-1}\xi_{ji}u_{i+1}^{T}u_{j} + 2\sum\limits_{j=1}^{i-1}\xi_{ji}f_{i}^{T}u_{j} .
	\end{align*}
	Since $\xi_{ji} = O(\sqrt{\epsilon})$ and $u_{l}^{T}u_{j}\leq\sqrt{\epsilon/(2k+1)}$ for $1\leq l \neq j \leq i+1$, we obtain 
	\begin{align*}
	& \ \ \ \ \| \sum\limits_{j=1}^{i-1}\xi_{ji}u_{j} \|^{2} +
	2\alpha_{i}\sum\limits_{j=1}^{i-1}\xi_{ji}u_{i}^{T}u_{j} + 2\beta_{i+1}\sum\limits_{j=1}^{i-1}\xi_{ji}u_{i+1}^{T}u_{j} + 2\sum\limits_{j=1}^{i-1}\xi_{ji}f_{i}^{T}u_{j} \\
	&= 2\sum\limits_{1\leq j<l \leq i-1}\xi_{ji}\xi_{li}u_{j}^{T}u_{l}+
	\sum\limits_{j=1}^{i-1}\xi_{ji}^{2}\| u_{j}\|^{2}+
	2\alpha_{i}\sum\limits_{j=1}^{i-1}O(\epsilon) +
	2\beta_{i+1}\sum\limits_{j=1}^{i-1}O(\epsilon)+
	2\sum\limits_{j=1}^{i-1}O(\epsilon\sqrt{\epsilon}) \\
	&= O(i\epsilon\sqrt{\epsilon}) + O(i\epsilon) + O[i(\alpha_{i}+\beta_{i+1})\epsilon] + O(i\epsilon\sqrt{\epsilon}) \\
	&=	O(i\epsilon) .
	\end{align*}
	Using the property of local orthogonality of $u_{i}$ , we have
	$$2\alpha_{i}\beta_{i+1}u_{i}^{T}u_{i+1} +
	2\alpha_{i}u_{i}^{T}f_{i}+2\beta_{i+1}u_{i+1}^{T}f_{i}+\|f_{i}\|^{2} = O(\bar{c}_{1}(m,n)\epsilon)$$
	with $\bar{c}_{1}(m,n)=c_{1}(m,n)+q_{1}(m,n)$.
	Thus
	\begin{equation*}
	\|\tilde{v}_{i}(1:m) \|^{2} = \alpha_{i}^{2}+\beta_{i+1}^{2} + O(\bar{c}_{1}(m,n)\epsilon) .
	\end{equation*}
	Using the similar method as above, from \eqref{3.4} we can obtain
	\begin{equation*}
	\|\tilde{v}_{i}(m+1:m+p)\|^{2} = \hat{\alpha}_{i}^{2} + \hat{\beta}_{i-1}^{2} 
	+  O(\bar{c}_{2}(p,n)\epsilon)
	\end{equation*}
	with $\bar{c}_{2}(p,n)=c_{2}(p,n)+q_{3}(p,n)$. Since
	$$1 = \|\tilde{v}_{i}\|^{2} = \|\tilde{v}_{i}(1:m) \|^{2}+\|\tilde{v}_{i}(m+1:m+p)\|^{2},$$
	we get 
	\begin{equation}\label{3.14}
	\alpha_{i}^{2}+\beta_{i+1}^{2}+\hat{\alpha}_{i}^{2}+\hat{\beta}_{i-1}^{2}
	= 1 + O(c_{3}(m,n,p) \epsilon) .
	\end{equation}
	
	For subdiagonal part, in finite precision arithmetic, we have  $\hat{\beta}_{i}=(\alpha_{i+1}\beta_{i+1}/\hat{\alpha}_{i})
	(1+\tau)$, where $|\tau|\leq \epsilon$
	\cite[\S 2.2]{Higham2002}, and thus
	$$\alpha_{i+1}\beta_{i+1} =\hat{\alpha}_{i}\hat{\beta}_{i}-\alpha_{i+1}\beta_{i+1}\tau .$$
	From \eqref{3.14} we have
	$$\alpha_{i+1}\beta_{i+1} \leq \frac{\alpha_{i+1}^{2}+\beta_{i+1}^2}{2}
	\leq \frac{2[1+O(c_{3}(m,n,p)\epsilon)]}{2}
	=1+O(c_{3}(m,n,p)\epsilon).$$
	Therefore, we obtain
	\begin{equation}\label{3.15}
	\alpha_{i+1}\beta_{i+1}=\hat{\alpha}_{i}\hat{\beta}_{i} +\gamma_{i}  ,
	\end{equation}
	where $|\gamma_{i}| \leq [1+O(c_{3}(m,n,p)\epsilon)]\epsilon = O(\epsilon)$.
	
	Combining \eqref{3.14} and \eqref{3.15}, we finally obtain \eqref{3.13}.
\end{proof}

We now show the connection between the process of computing $\widehat{U}_{k}$, $\widehat{V}_{k}$, $\widehat{B}_{k}$ and the upper Lanczos bidiagonalization of $Q_{L}$ in finite precision arithmetic.
\begin{theorem}\label{Th3.2}
	For the $k$-step JBD process with the semiorthogonalization strategy, the following relation holds:
	\begin{equation}\label{3.16}
	Q_{L}^{T}\widehat{U}_{k}=\widehat{V}_{k}(\widehat{B}_{k}^{T}+\widehat{D}_{k}) +\hat{\beta}_{k}\hat{v}_{k+1}e_{k}^{T}+\widehat{G}_{k} ,
	\end{equation}
	where $\widehat{D}_{k}$ is upper triangular with zero diagonals, and
	\begin{equation}\label{3.17}
	\| \widehat{G}_{k}\|
	= O(c_{4}(m,n,p)\|\widehat{B}_{k}^{-1}\|\epsilon) ,
	\end{equation}
	with $c_{4}(m,n,p)=c_{1}(m,n)+c_{2}(p,n)+\bar{q}({m,n,p})$.
\end{theorem}
\begin{proof}
	Combining \eqref{3.8} and \eqref{3.9}, we have
	\begin{align*}
	Q_{A}^{T}Q_{A}V_{k} &=Q_{A}^{T}U_{k+1}(B_{k}+C_{k})+Q_{A}^{T}F_{k} \\
	&= [V_{k}(B_{k}^{T}+D_{k})+\alpha_{k+1}v_{k+1}e_{k+1}^{T}+G_{k+1}](B_{k}+
	C_{k})+Q_{A}^{T}F_{k} \\
	&= V_{k}B_{k}^{T}B_{k}+\alpha_{k+1}\beta_{k+1}v_{k+1}e_{k}^{T}+
	V_{k}(B_{k}^{T}+D_{k})C_{k}+V_{k}D_{k}B_{k}+ \\
	& \ \ \ \ \ \  G_{k+1}(B_{k}+C_{k})+ Q_{A}^{T}F_{k} .
	\end{align*}
	Premultiply \eqref{3.10} by $Q_{L}^{T}$, we have
	\begin{align*}
	Q_{L}^{T}Q_{L}V_{k}=[Q_{L}^{T}\widehat{U}_{k}(\widehat{B}_{k}+\widehat{C}_{k})+
	Q_{L}^{T}\widehat{F}_{k}]P .
	\end{align*}
	Adding the above two equalities, we obtain
	\begin{align*}
	& \ \ \ \ (Q_{A}^{T}Q_{A}+Q_{L}^{T}Q_{L})V_{k} \\
	&=  V_{k}[I_{k}-P\widehat{B}_{k}^{T}\widehat{B}_{k}P+H_{k}]
	+Q_{L}\widehat{U}_{k}\widehat{B}_{k}^{T}P+(\hat{\alpha}_{k}\hat{\beta}_{k}+\gamma_{k})v_{k+1}e_{k}^{T} + V_{k}D_{k}B_{k}+ \\
	& \ \ \ \ V_{k}(B_{k}^{T}+D_{k})C_{k} + Q_{L}^{T}\widehat{U}_{k}\widehat{C}_{k}P +
	G_{k+1}(B_{k}+C_{k})+ Q_{A}^{T}F_{k}+Q_{L}^{T}\widehat{F}_{k}P .
	\end{align*}
	Since $(Q_{A}^{T}Q_{A}+Q_{L}^{T}Q_{L})V_{k}=V_{k}$, after some rearrangement we obtain 
	\begin{align*}
	\widehat{V}_{k}\widehat{B}_{k}^{T}\widehat{B}_{k}
	&= Q_{L}^{T}\widehat{U}_{k}\widehat{B}_{k}-\hat{\alpha}_{k}\hat{\beta}_{k}\hat{v}_{k+1}e_{k}^{T}
	+\bar{E}_{1} + \bar{E}_{2} ,
	\end{align*}
	where
	\begin{align*}
	\bar{E}_{1}= \widehat{V}_{k}P[D_{k}B_{k}+(B_{k}^{T}+D_{k})C_{k}]P+ Q_{L}^{T}\widehat{U}_{k}\widehat{C}_{k} ,
	\end{align*}
	and
	\begin{align*}
	\bar{E}_{2}
	&= [G_{k+1}(B_{k}+C_{k})+Q_{A}^{T}F_{k} + Q_{L}^{T}\widehat{F}_{k}P +V_{k}H_{k}]P - \gamma_{k}\hat{v}_{k+1}e_{k}^{T} .
	\end{align*}
	
	According to the structure of matrices $C_{k}$ and $D_{k}$, with simple calculation we can verify that $P[D_{k}B_{k}+(B_{k}^{T}+D_{k})C_{k}]P$ is an upper triangular matrix with zero diagonals, which is denoted by $Y_{k}$. Noticing that the $i$-th column of $Q_{L}^{T}\widehat{U}_{k}\widehat{C}_{k}$ is
	$\sum_{j=1}^{i-2}\hat{\xi}_{ji}Q_{L}^{T}\hat{u}_{j}$, by Lemma \ref{Lem3.1}, there exit
	coefficients $\rho_{1i}, \dots, \rho_{i-1,i}$ such that 
	$$ \sum_{j=1}^{i-2}\hat{\xi}_{ji}Q_{L}^{T}\hat{u}_{j} =  \sum_{j=1}^{i-1}\rho_{ji}\hat{v}_{j}+O(\bar{q}({m,n,p})\epsilon) .$$
	Therefore, we have
	$$ Q_{L}^{T}\widehat{U}_{k}C_{k}=\widehat{V}_{k}W_{k}+O(\bar{q}({m,n,p})\epsilon) ,$$
	where 
	$$ W_{k} = \begin{pmatrix}
	0 & \rho_{12} & \rho_{13} & \cdots & \rho_{1k} \\
	& 0  & \rho_{23} & \cdots & \rho_{2k} \\
	&    & \ddots    & \ddots & \vdots \\
	&    &     & \ddots & \rho_{k-1,k} \\
	&    &     &  & 0 
	\end{pmatrix} \in \mathbb{R}^{k\times k} $$
	is upper triangular with zero diagonals. 
	
	Notice that $\|V_{k}\| \leq \sqrt{1+\xi(V_{k})}=1+O(\sqrt{\epsilon})$. From \eqref{3.14}, we can get
	$$ \|B_{k}\| \leq \sqrt{2}\max_{1\leq i \leq k}(\alpha_{i}^{2}+\beta_{i+1}^{2})^{1/2}
	\leq \sqrt{2}+O(c_{3}(m,n,p)\epsilon) .\footnote{Here we use the result of an exercise from \cite[Chapter6, Problems 6.14]{Higham2002}, which gives the upper bound of the $p$-norm of a row/column sparse matrix.}$$ Similar to $\|B_{k}\|$, by Theorem \ref{Th3.1}, we can get 
	\begin{equation}\label{3.18}
	\|H_{k}\|=O(c_{3}(m,n,p)).
	\end{equation}
	Using these upper bounds, with simple but tedious calculation, we can prove that 
	$$\|\bar{E}_{2}\| = O(c_{4}(m,n,p)\epsilon) .$$
	
	From the above, we obtain
	\begin{align*}
	Q_{L}^{T}\widehat{U}_{k}-\widehat{V}_{k}\widehat{B}_{k}^{T} -\hat{\beta}_{k}\hat{v}_{k+1}e_{k}^{T} 
	= -\widehat{V}_{k}(W_{k}+Y_{k})\widehat{B}_{k}^{-1} - [\bar{E}_{2}+ O(\bar{q}({m,n,p})\epsilon)]\widehat{B}_{k}^{-1}  .
	\end{align*}
	Noticing that $-(W_{k}+Y_{k})\widehat{B}_{k}^{-1}$ is upper triangular with zero diagonals, which is denoted by $\widehat{D}_{k}$, we finally obtain
	\begin{align*}
	Q_{L}^{T}\widehat{U}_{k}=\widehat{V}_{k}(\widehat{B}_{k}^{T}+\widehat{D}_{k}) +\hat{\beta}_{k}\hat{v}_{k+1}e_{k}^{T}+\widehat{G}_{k} ,
	\end{align*}
	where $\widehat{G}_{k}=-[\bar{E}_{2}+ O(\bar{q}({m,n,p})\epsilon)]\widehat{B}_{k}^{-1}$ and $\| \widehat{G}_{k}\|= O(c_{4}(m,n,p)\|\widehat{B}_{k}^{-1}\|\epsilon)$.
\end{proof}

If we write the matrix $\widehat{D}_{k}$ as 
$$\widehat{D}_{k} = \begin{pmatrix}
0  & \hat{\eta}_{12} &\hat{\eta}_{13} &\cdots &\hat{\eta}_{1k} \\
&0 & \hat{\eta}_{23} &\cdots &\hat{\eta}_{2k} \\
& & \ddots &\ddots &\vdots \\
& & & 0 & \hat{\eta}_{k-1,k} \\
& & & & 0
\end{pmatrix} \in \mathbb{R}^{k\times k} ,$$
then for each $i = 1, \dots, k$, from \eqref{3.16} we have
$$\hat{\beta}_{i}\hat{v}_{i+1} = Q_{L}^{T}\hat{u}_{i}-\hat{\alpha}_{i}\hat{v}_{i}-
\sum_{j=1}^{i-1}\hat{\eta}_{ji}\hat{v}_{j} - \hat{g}_{i} ,$$
where $\|\hat{g}_{i}\|= O(c_{4}(m,n,p)\|\widehat{B}_{k}^{-1}\|\epsilon)$, which corresponds to the reorthogonalization of $\hat{v}_{i}$ with error term $\hat{g}_{i}$. Therefore, combining \eqref{3.10} and \eqref{3.16}, we can treat the process of computing $\widehat{U}_{k}$, $\widehat{V}_{k}$ and $\widehat{B}_{k}$ as the upper Lanczos bidiagonalization of $Q_{L}$ with the semiorthogonalization strategy within error $\delta= O(c_{4}(m,n,p)\|\widehat{B}_{k}^{-1}\|\epsilon)$.

By \eqref{3.8} and \eqref{3.9}, we can treat the process of computing $U_{k+1}$, $V_{k}$ and $B_{k}$ as the lower Lanczos bidiagonalization of $Q_{A}$ with the semiorthogonalization strategy. Therefore, the computed $B_{k}$ is, up to roundoff, the Ritz-Galerkin projection of $Q_{A}$ on the subspace $span(U_{k+1})$ and $span(V_{k})$, i.e., we have the following result.
\begin{theorem}\label{Th3.3}
For the $k$-step JBD process with the semiorthogonalization strategy, suppose that the compact $QR$ factorizations of $U_{k}$ and $V_{k}$ are $U_{k}=M_{k}R_{k}$ and $V_{k}=N_{k}S_{k}$, where the diagonals of the upper triangular matrices $R_{k}$ and $S_{k}$ are nonnegative. Then
\begin{equation}\label{3.19}
M_{k}^{T}Q_{A}N_{k}=B_{k}+E_{k} ,
\end{equation}
where the elements of $E_{k}$ are of $O(\tilde{q}(m,n,p)\epsilon)$ with $\tilde{q}(m,n,p)=q_{1}(m,n)+q_{2}(p,n)$.	
\end{theorem}

Since the $k$-step Lanczos bidiagonalization is equivalent to the $(2k+1)$-step symmetric Lanczos process \cite[\S 7.6.1]{Bjorck1996}, using the method appeared in \cite[\S 7.6.1]{Bjorck1996}, Theorem \ref{Th3.3} can be concluded from \cite[Theorem 5]{Simon1984b}. By Wielandt-Hoffman theorem \cite[Theorem 8.6.4]{Golub2012}, the singular values of $B_{k}$ are, up to error $O(\tilde{q}(m,n,p)\epsilon)$, the singular values of $Q_{A}$. Therefore, Theorem \ref{Th3.3} means that if we use the SVD of $B_{k}$ to approximate some generalized singular values of $\{A,L\}$, the ``ghosts" can be avoided from appearing and the final accuracy of approximated quantities is close to the machine precision.

In \cite{Zha1996}, the author suggests that one can also use the SVD of $\widehat{B}_{k}$ to approximate some generalized singular values and vectors of $\{A,L\}$. Similar to the above theorem, combining \eqref{3.10} and \eqref{3.16}, we can obtain the following result.
\begin{theorem}\label{Th3.4}
	For the $k$-step JBD process with the semiorthogonalization strategy, suppose that the compact QR factorizations of $\widehat{U}_{k}$ and $\widehat{V}_{k}$ are
	$\widehat{U}_{k}=\widehat{M}_{k}\widehat{R}_{k}$ and $\widehat{V}_{k}=\widehat{N}_{k}\widehat{S}_{k}$, where the diagonals of the upper triangular matrices $\widehat{R}_{k}$ and $\widehat{S}_{k}$ are nonnegative. Then
	\begin{equation}\label{3.20}
		\widehat{M}_{k}^{T}Q_{L}\widehat{N}_{k}=\widehat{B}_{k}+\widehat{E}_{k} ,
	\end{equation}
	where the elements of $\widehat{E}_{k}$ are of $\delta=O(c_{4}(m,n,p)\|\widehat{B}_{k}^{-1}\|\epsilon)$.
\end{theorem}

Theorem \ref{Th3.4} indicates that the computed $\widehat{B}_{k}$ is approximate the Ritz-Galerkin projection of $Q_{L}$ on the subspace $span(\widehat{U}_{k})$ and $span(\widehat{V}_{k})$ within error $O(\delta)$. Therefore, if we use the SVD of $\widehat{B}_{k}$ to approximate some generalized singular values $\{A,L\}$, the ``ghosts" can be avoided from appearing and the final accuracy of approximated quantities is high enough, as long as $\|\widehat{B}_{k}^{-1}\|$ does not become too large.

\section{ The JBD process with partial reorthogonalization }\label{sec4}
In order to implement the semiorthogonalization strategy, we need to decide when to reorthogonalize, and which Lanczos vectors are necessary to include in the reorthogonalization step. By the analysis in the previous section, the process of computing $U_{k+1}$, $V_{k}$ and $B_{k}$ can be treated as the lower Lanczos bidiagonalization of $Q_{A}$, so our reorthogonalization strategy can be based on the partial reorthogonalization of $u_{i}$ and $v_{i}$; see \cite{Simon1984a,Larsen1998}. The central idea is that the levels of orthogonality of $u_{i}$ and $v_{i}$ satisfy the following coupled recurrences \cite[Theorem 6]{Larsen1998}.
 
\begin{theorem}\label{Th4.1}
	Let $\mu_{ji}=u_{j}^{T}u_{i}$, $\nu_{ji}=v_{j}^{T}v_{i}$
	and $\mu_{j0} \equiv 0$, $\nu_{j0} \equiv 0$. Then $\mu_{jj}=1$ for $1 \leq j \leq i+1$
	and $\nu_{jj}=1$ for $1 \leq j \leq i$, while
	\begin{align}\label{4.1}
		\beta_{i+1}\mu_{j,i+1}=\alpha_{j}\nu_{ji}+\beta_{j}\nu_{j-1,i}-\alpha_{i}\mu_{ji}-u_{j}^{T}f_{i}+v_{i}^{T}g_{j}
	\end{align}
	for $1 \leq j \leq i$, and 
	\begin{align}\label{4.2}
		\alpha_{i}\nu_{ji}=\beta_{j+1}\mu_{j+1,i}+\alpha_{j}\mu_{ji}-\beta_{i}\nu_{j,i-1}+u_{i}^{T}f_{j}-v_{j}^{T}g_{i}
	\end{align}
	for $1 \leq j \leq i-1$.
\end{theorem}

Theorem \ref{Th4.1} shows that the inner products $u_{j}^{T}u_{i+1}$ and $v_{j}^{T}v_{i}$ are simply linear combinations of the inner products from the previous Lanczos vectors, thus we can estimate quantities $\mu_{j,i+1}$ and $\nu_{ji}$ if we have proper estimations of $|u_{i}^{T}f_{j}-v_{j}^{T}g_{i}|$ and $|u_{j}^{T}f_{i}-v_{i}^{T}g_{j}|$. The two quantities $|u_{i}^{T}f_{j}-v_{j}^{T}g_{i}|$ and $|u_{j}^{T}f_{i}-v_{i}^{T}g_{j}|$ are about $O(\epsilon)$ and accurate estimates of them have been discussed in detail in \cite{Larsen1998}. Since $\tilde{v}_{i}^{T}\tilde{v}_{j}\approx v_{i}^{T}v_{j}$, the estimated $\nu_{ji}$ is also a good estimate of $\tilde{v}_{j}^{T}\tilde{v}_{i}$. Therefore, using these estimates, we can monitor the loss of orthogonality of Lanczos vectors $u_{i}$ and $\tilde{v}_{i}$ directly without forming inner products, which enables us to determine when and against which of the previous Lanczos vectors to reorthogonalize. 

On the other hand, it has been shown in \cite{Li2019} that the orthogonality level of $\widehat{U}_{k}$ is affected not only by those of $U_{k+1}$ and $\widetilde{V}_{k}$, but also by a factor $\|\widehat{B}_{k}^{-1}\|$. Therefore, if $\widehat{B}_{k}$ is not very ill-conditioned, the orthogonality of $\widehat{U}_{k}$ will not be too bad even if we only reorthogonalize $u_{i}$ and $\tilde{v}_{i}$ but not $\hat{u}_{i}$. From the above discussions, we finally obtain the JBD process with partial reorthogonalization, which is described in Algorithm \ref{alg2}.

\begin{algorithm}[htb]
	\caption{The $k$-step JBDPRO}
	\begin{algorithmic}[1]\label{alg2}
		\STATE {Choosing a starting vector $b \in \mathbb{R}^{m}$,
			$\beta_{1}u_{1}=b,\ \beta_{1}=\| b\| $ }
		\STATE {$\alpha_{1}\tilde{v}_{1}=QQ^{T}\begin{pmatrix}
			u_{1} \\
			0_{p}
			\end{pmatrix} $}
		\STATE {  $\hat{\alpha}_{1}\hat{u}_{1}=\tilde{v}_{1}(m+1:m+p) $}
		\FOR{$i=1,2,\ldots,k,$}
		\STATE $r_{i+1}=\tilde{v}_{i}(1:m)-\alpha_{i}u_{i} $
		\STATE Update $\mu_{ji}\rightarrow \mu_{ji+1}, \ j=1,\cdots, i$
		\STATE Determine a set of indices $T_{i}\subseteq \{ j|1\leq j \leq i-1\}$
		\FOR{$j \in T_{i}$}
		\STATE $r_{i+1} = r_{i}-(u_{j}^{T}r_{i+1})u_{j}$
		\STATE	Reset $\mu_{j,i+1}$ to $O(\epsilon)$
		\ENDFOR
		\STATE $\beta_{i+1}u_{i+1} = r_{i+1}$
		\STATE $p_{i+1} =
		QQ^{T}\begin{pmatrix}
		u_{i+1} \\
		0_{p}
		\end{pmatrix}-\beta_{i+1}\tilde{v}_{i} $
		\STATE Update $\nu_{ji}\rightarrow \nu_{ji+1}, \ j=1,\cdots, i$
		\STATE Determine a set of indices $S_{i}\subseteq \{ j|1\leq j \leq i-1\}$
		\FOR{$j \in S_{i}$}
		\STATE $p_{i+1}=p_{i+1}-(v_{j}^{T}p_{i+1})v_{j}$
		\STATE Reset $\nu_{j,i+1}$ to $O(\epsilon)$
		\ENDFOR
		\STATE $\alpha_{i+1}\tilde{v}_{i+1}=p_{i}$
		\STATE $\hat{\beta}=(\alpha_{i+1}\beta_{i+1})/\hat{\alpha}_{i} $
		\STATE $\hat{\alpha}_{i+1}\hat{u}_{i+1}=
		(-1)^{i}\tilde{v}_{i+1}(m+1:m+p)-\hat{\beta}_{i}\hat{u}_{i} $
		\ENDFOR
	\end{algorithmic}
\end{algorithm}

In Algorithm \ref{alg2}, we need to determine two sets $T_{i}$ and $S_{i}$ at each iteration. The methods of choosing which previous Lanczos vectors to reorthogonalize have been discussed in detail by Simon \cite{Simon1984a} and Larsen \cite{Larsen1998}, for symmetric Lanczos process and Lanczos bidiagonalization, respectively. They introduce the $\eta$-criterion. Here we use the reorthogonalization of $u_{i+1}$ to explain it. At the $i$-th iteration, we only need to reorthogonalize against the vectors where $\mu_{ji+1}$ is larger than some constant $\eta$ satisfying $\epsilon < \eta < \omega_{0}$. It is sufficient to choose the vectors where $\mu_{ji+1}$ exceeds $\omega_{0}$ and their neighbors exceed $\eta$ to be included in the reorthogonalization step, while a few isolated components that exceeding $\eta$ are quite harmless \cite{Simon1984a,Larsen1998}. Therefore, the indices sets $T_{i}$ and $S_{i}$ can be described by the formulas
\begin{align}
	& T_{i} = \bigcup_{\mu_{j,i+1}>\omega_{0}} \{l | 1\leq j-r\leq l \leq j+s \leq i-1, \mu_{li+1} > \eta \} \label{4.3}  ,\\
	& S_{i} = \bigcup_{\nu_{j,i+1}>\omega_{0}} \{l | 1\leq j-r\leq l \leq j+s \leq i-1, \nu_{li+1} > \eta \} \label{4.4}  .
\end{align}

Simon \cite{Simon1984a} demonstrates that using the $\eta$-criterion in partial reorthogonalization could significantly reduces the amount of extra reorthogonalization work. Experimentally he finds that $\eta=\epsilon^{3/4}$ is the value that minimizes the total amount of reorthogonalization work for the symmetric Lanczos process. In Algorithm \ref{alg2}, we also choose $\eta=\epsilon^{3/4}$ to implement the partial reorthogonalization. 

For the JBDPRO algorithm with $\eta$-criterion, the orthogonality levels of $u_{i}$ and $\tilde{v}_{i}$ will be $O(\eta)$. By using the same method appeared in the proof of Lemma \ref{Lem3.2}, we can prove that $D_{k} = O(\eta)$ and $C_{k} = O(\eta)$. Notice that we do not reorthogonalize $\hat{u}_{i}$, which can save a big amount of reorthogonalization work. The following theorem says that if $\widehat{B}_{k}$ is not very ill-conditioned, the orthogonality of $\widehat{U}_{k}$ will be at a desired level.
\begin{theorem}\label{Th4.2}
	For the $k$-step JBDPRO algorithm, the orthogonality level of $\widehat{U}_{k}$ satisfies
	\begin{align}\label{4.5}
	\xi(\widehat{U}_{k})= O(\|\widehat{B}_{k}^{-1}\|^{2}\eta) .
	\end{align}
\end{theorem}
\begin{proof}
	Since we do not reorthogonalize any $\hat{u}_{i}$, which means that $\widehat{C}_{k}=0$, by \eqref{3.10} we have
	$$\widehat{B}_{k}^{T}\widehat{U}_{k}^{T}\widehat{U}_{k}\widehat{B}_{k} = (Q_{L}\widehat{V_{k}}-\widehat{F}_{k})^{T}
	(Q_{L}\widehat{V_{k}}-\widehat{F}_{k}) ,$$
	and
	\begin{align}\label{4.6}
		\begin{split}
			& \ \ \ \ \ \widehat{B}_{k}^{T}(I_{k}-\widehat{U}_{k}^{T}\widehat{U}_{k})\widehat{B}_{k}
			=\widehat{B}_{k}^{T}\widehat{B}_{k}-(Q_{L}\widehat{V}_{k}-\widehat{F})^{T}(Q_{L}\widehat{V}_{k}-\widehat{F}_{k}) \\
			& = I_{k}-PB_{k}^{T}B_{k}P+H_{k}-\widehat{V}_{k}^{T}Q_{L}^{T}Q_{L}\widehat{V}_{k}+ 
			 \widehat{V}_{k}^{T}Q_{L}^{T}\widehat{F}_{k}+\widehat{F}_{k}^{T}Q_{L}\widehat{V}_{k}-
			\widehat{F}_{k}^{T}\widehat{F}_{k} \\
			& = I_{k}-PB_{k}^{T}B_{k}P - PV_{k}^{T}(I_{k}-Q_{A}^{T}Q_{A})V_{k}P+ 
			\widehat{V}_{k}^{T}Q_{L}^{T}\widehat{F}_{k}+\widehat{F}_{k}^{T}Q_{L}\widehat{V}_{k}-
			\widehat{F}_{k}^{T}\widehat{F}_{k}+ H_{k} .
		\end{split}
	\end{align}
	By \eqref{3.8}, we have 
	\begin{align}\label{4.7}
		\begin{split}
			V_{k}^{T}Q_{A}^{T}Q_{A}V_{k}
			& = [U_{k+1}(B_{k}+D_{k})+F_{k}]^{T}[U_{k+1}(B_{k}+D_{k})+F_{k}] \\
			& = B_{k}^{T}U_{k+1}^{T}U_{k+1}B_{k} + \bar{E}_{3} ,
		\end{split}
	\end{align}
	where 
	\begin{align*}
		\bar{E}_{3}
		& = D_{k}^{T}U_{k+1}^{T}U_{k+1}B_{k} + B_{k}^{T}U_{k+1}^{T}U_{k+1}D_{k} + (B_{k}+D_{k})^{T}U_{k+1}^{T}F_{k}+ \\
		& \ \ \ \ \ F_{k}^{T}U_{k+1}(B_{k}+D_{k}) + D_{k}^{T}U_{k+1}^{T}U_{k+1}D_{k} +F_{k}^{T}F_{k}  .
	\end{align*}
	Since $D_{k} = O(\eta)$ , with simple calculation we can obtain
	$$\|\bar{E}_{3}\|= O(\eta) .$$
	Substituting \eqref{4.7} into \eqref{4.6}, we have
	\begin{align*}
		\widehat{B}_{k}^{T}(I_{k}-\widehat{U}_{k}^{T}\widehat{U}_{k})\widehat{B}_{k}
		& = (I_{k}-\widehat{V}_{k}^{T}\widehat{V}_{k}) - PB_{k}^{T}(I_{k+1}-U_{k+1}^{T}U_{k+1}) B_{k}P+ \\
		& \ \ \ \ \  \widehat{V}_{k}^{T}Q_{L}^{T}\widehat{F}_{k}+\widehat{F}_{k}^{T}Q_{L}\widehat{V}_{k}-
		\widehat{F}_{k}^{T}\widehat{F}_{k}+ H_{k} + P\bar{E}_{3}P .
	\end{align*}
	With simple calculation we can obtain
	$$\|\widehat{V}_{k}^{T}Q_{L}^{T}\widehat{F}_{k}+\widehat{F}_{k}^{T}Q_{L}\widehat{V}_{k}-
		\widehat{F}_{k}^{T}\widehat{F}_{k}+H_{k}\| = O(c_{3}(m,n,p)\epsilon) .$$
	Therefore,
	\begin{align*}
		\widehat{B}_{k}^{T}(I_{k}-\widehat{U}_{k}^{T}\widehat{U}_{k})\widehat{B}_{k}
		& = (I_{k}-\widehat{V}_{k}^{T}\widehat{V}_{k}) - PB_{k}^{T}(I_{k+1}-U_{k+1}^{T}U_{k+1}) B_{k}P + O(\eta) .
	\end{align*}
	
	Notice that in the JBDPRO algorithm, we have
	$\xi(\widehat{V}_{k})=\|I_{k}-\widehat{V}_{k}^{T}\widehat{V}_{k}\|= O(\eta)$ and
	$\xi(U_{k+1})=\|I_{k+1}-U_{k+1}^{T}U_{k+1}\| = O(\eta)$. We finally obtain
	$$\xi(\widehat{U}_{k})
		= \|I_{k}-\widehat{U}_{k}^{T}\widehat{U}_{k}\|
		 \leq \|\widehat{B}_{k}^{-1}\|^{2}[ \|B_{k}\|^{2}O(\eta)+O(\eta)] 
		= O(\|\widehat{B}_{k}^{-1}\|^{2}\eta) ,$$
	which is the desired result.	
\end{proof}

Since the orthogonality level of $\hat{v}_{i}$ is $O(\eta)$, which is below $\sqrt{\delta/(2k+1)}$, by Theorem \ref{Th3.4}, the relation \eqref{3.20} holds as long as $\kappa(\widehat{U}_{k})$ is below $\sqrt{\delta/(2k+1)}$, i.e.,
the following condition should be satisfied:
\begin{equation*}
\|\widehat{B}_{k}^{-1} \|^{2}\epsilon^{3/4}  \lesssim 
\sqrt{\delta/(2k+1)}  ,
\end{equation*}
which leads to 
\begin{equation}\label{4.8}
\|\widehat{B}_{k}^{-1} \|^{3}  \lesssim \dfrac{c_{4}(m,n,p)}{(2k+1)\sqrt{\epsilon}} .
\end{equation}
Therefore, for the JBDPRO algorithm, if we use $\widehat{B}_{k}$ to approximate some generalized singular values of $\{A,L\}$, the ``ghosts" can be avoided from appearing and the final accuracy of approximated quantities is high enough, as long as the growth of $\|\widehat{B}_{k}^{-1}\|$ can be controlled by \eqref{4.8}. 
 
\section{Numerical experiments}\label{sec5}
In this section, we provide several numerical examples to illustrate our theory about the properties of the JBD process with the semiorthogonalization strategy and the JBDPRO algorithm. The matrices are constructed by ourselves or chosen from the University of Florida Sparse Matrix Collection \cite{Davis2011}. For the first pair, the matrices $A$ and $L$, which are denoted by $A_{c}$ and $L_{s}$, respectively, are constructed by ourselves. Let $n=800$ and $C=diag(c)$, where $c=(\dfrac{3n}{2}, \dfrac{3n}{2}-1, \dots, \dfrac{n}{2}+1)/2n$. Then let $s = ((1-c_{1}^{2})^{1/2}, \dots, (1-c_{n}^{2})^{1/2})$ and $S=diag(s)$. Let $D$ be the matrix generated by the MATLAB built-in function $\texttt{D=gallery(`orthog',n,2)}$, which means that $D$ is a symmetric orthogonal matrix. Finally, let $A=CD$ and $L=SD$. For the second pair, $A$ and $L$ are the square matrices {\sf dw2048} and {\sf rdb2048} from electromagnetics problems and computational fluid dynamics problems, respectively. For the third pair, $A$ is the square matrix {\sf ex31} from computational fluid dynamics problems, $L_{m}=diag(l)$, where $l=(3m, 3m-1, \dots, 2m+1)/4000$ and $m$ is the row number of $A$. For the fourth pair, $A$ is the square matrix {\sf rdb5000} from computational fluid dynamics problems and $L=L_{1}$, which is the discrete approximation of the first order derivative operator. The properties of our test matrices are described in table \ref{tab1}, where $cond(\cdot)$ means the condition number of a matrix.
\begin{equation}\label{5.1}
	L_1 = \left(
	\begin{array}{ccccc}
		1 & -1 &  &  &  \\
		& 1 & -1 &  &  \\
		&  & \ddots & \ddots &  \\
		&  &  & 1  & -1\\
	\end{array}
	\right)\in \mathbb{R}^{(n-1)\times n},
\end{equation}

\begin{table}[htp]
	\centering
	\caption{Properties of the test matrices.}
	\begin{tabular}{|l|l|l|l|l|l|}
		\hline 
		$A$ &$m\times n$ &$cond(A)$ &$L$  &$p\times n$  &$cond(L)$ \\  \hline  
		$A_{c}$	 &$800\times 800$	&2.99 &$L_{s}$ &$800\times 800$  &1.46 	\\  
		{\sf rdb2048} &$2048\times 2048$  &2026.80 &{\sf dw2048}	 &$2048\times 2048$	&5301.50  \\  
		{\sf ex31}  &$3909\times 3909$  &$1.01\times 10^{6}$  &$L_{m}$  &$3909\times 3909$ & 1.50 \\  
		{\sf rdb5000}  &$5000\times 5000$   & 4304.90  &$L_{1}$  &$4999\times5000$   & 3183.1 \\ \hline  
	\end{tabular}
	\label{tab1}
\end{table}

The numerical experiments are performed on an Intel (R) Core (TM) i7-7700 CPU 3.60GHz with the main memory 8GB using the Matlab R2017b with the machine precision $\epsilon = 2.22 \times 10^{-16}$ under the Windows 10 operating system.  For each matrix pair $\{A, L\}$, we use $b=(1,\dots,1)^{T}\in\mathbb{R}^{m}$ as the starting vector of the JBD process, where $m$ is the row number of $A$. We mention that our results are based on the assumption that the inner least squares problem \eqref{2.1} is solved accurately at each step. Therefore, for the JBD process in the numerical experiments, the $QR$ factorization of 
$\begin{pmatrix}
A \\ L
\end{pmatrix}$
is computed, and $QQ^{T}\tilde{u}_{i}$ is computed explicitly using $Q$ at each step. 

In the JBD process, in order to make sure that $\widetilde{V}_{k}$ does not deviate far from the column space of $Q$, the order of the matrices in the matrix pair $\{A,L\}$ may need to be adjusted; see \eqref{2.14} and Remark \ref{rem2.1}. Especially, in the four test examples, we implement the JBD process of {\sf \{$L_{m}$,ex31\}} instead of {\sf \{ex31,$L_{m}$\}}.

\begin{figure}[htp]
	\begin{minipage}{0.48\linewidth}
		\centerline{\includegraphics[width=5.5cm,height=3.5cm]{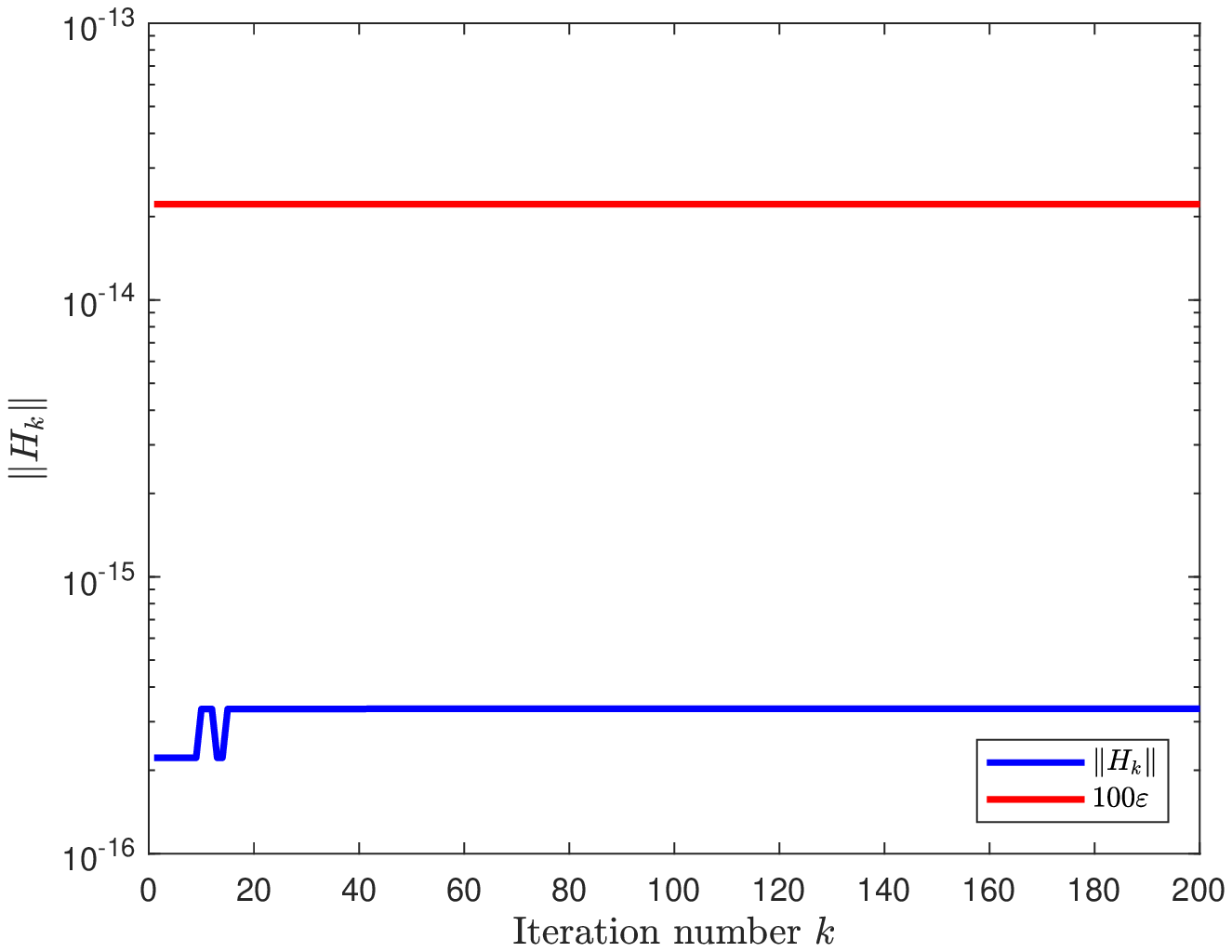}}
		\centerline{(a)}
	\end{minipage}
	\hfill
	\begin{minipage}{0.48\linewidth}
		\centerline{\includegraphics[width=5.5cm,height=3.5cm]{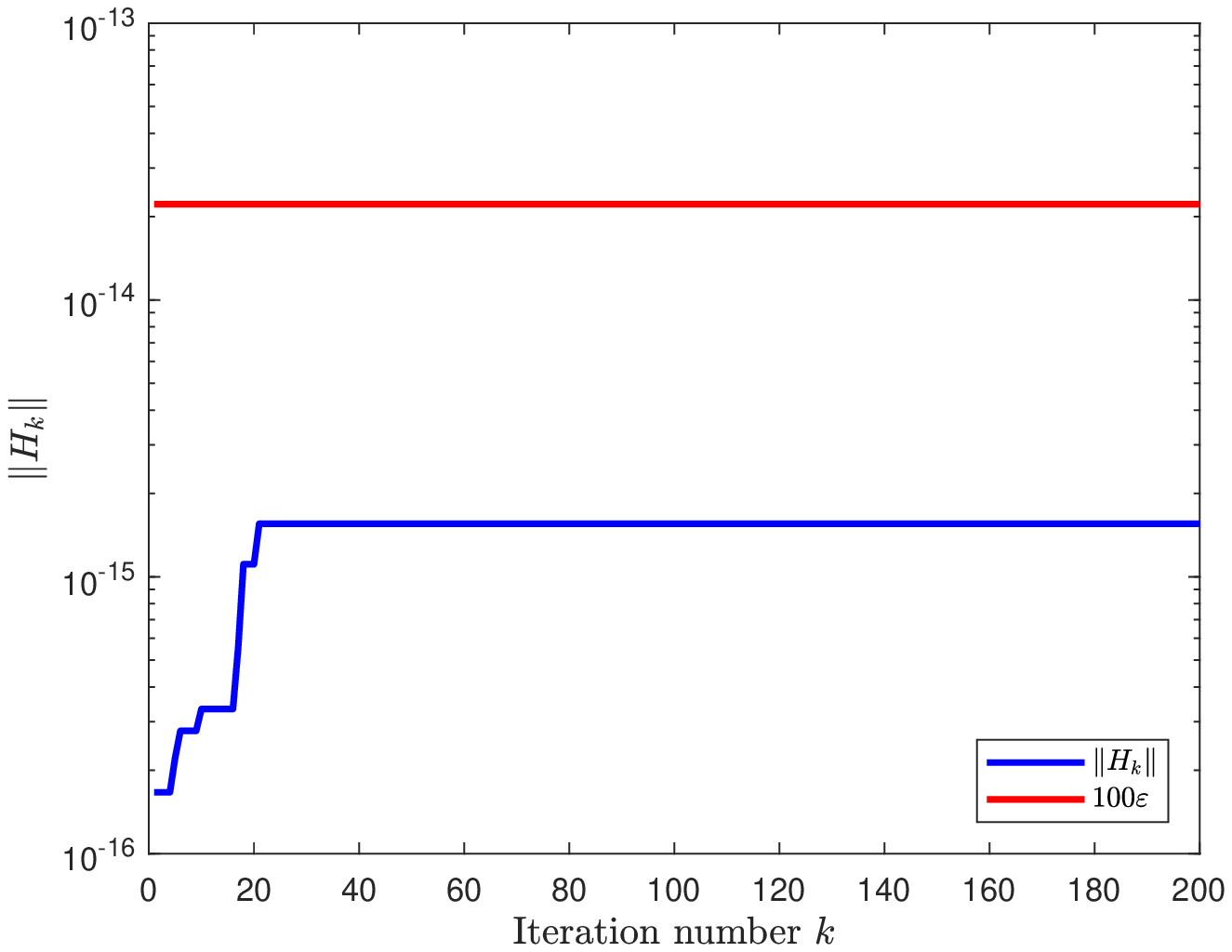}}
		\centerline{(b)}
	\end{minipage}
	
	\vfill
	\begin{minipage}{0.48\linewidth}
		\centerline{\includegraphics[width=5.5cm,height=3.5cm]{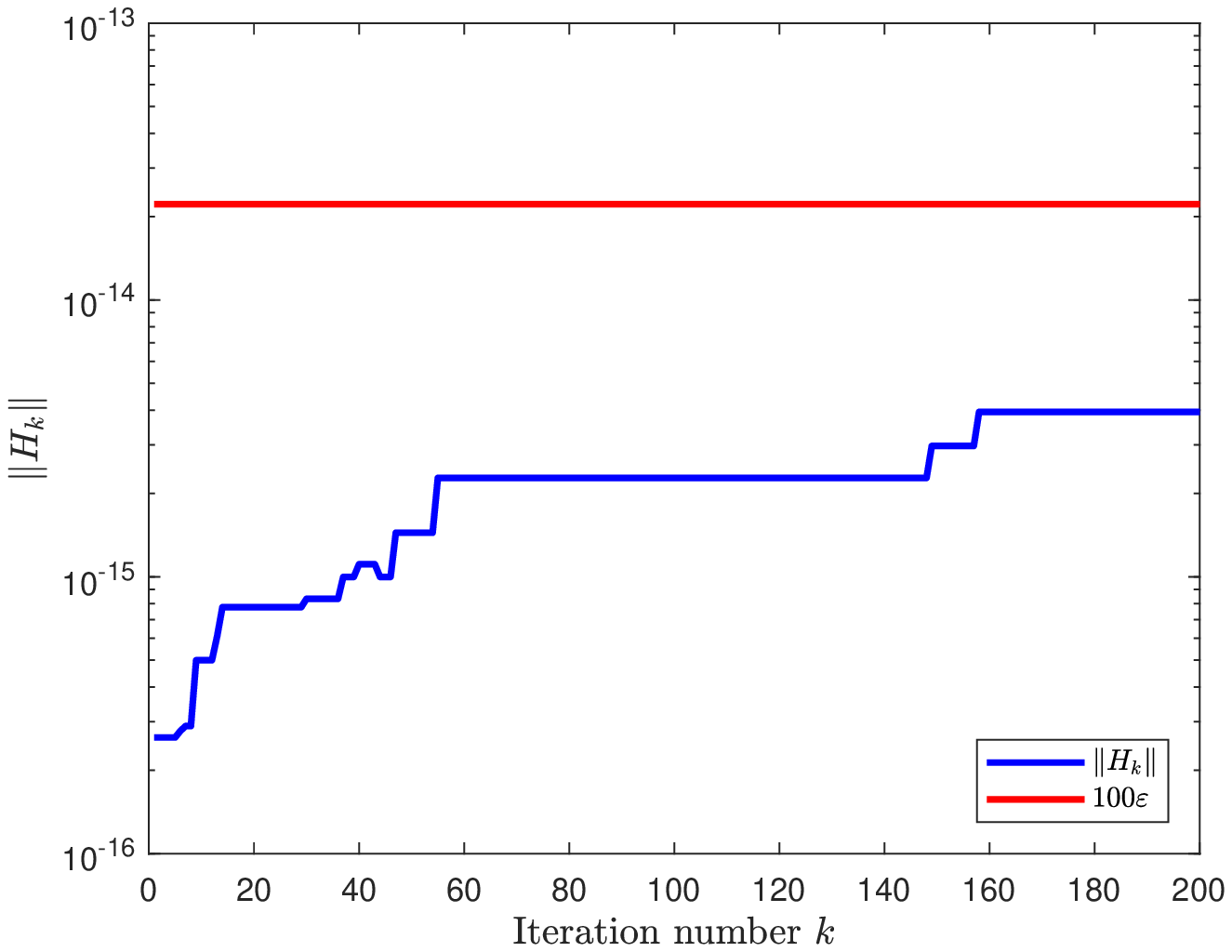}}
		\centerline{(c)}
	\end{minipage}
	\hfill
	\begin{minipage}{0.48\linewidth}
		\centerline{\includegraphics[width=5.5cm,height=3.5cm]{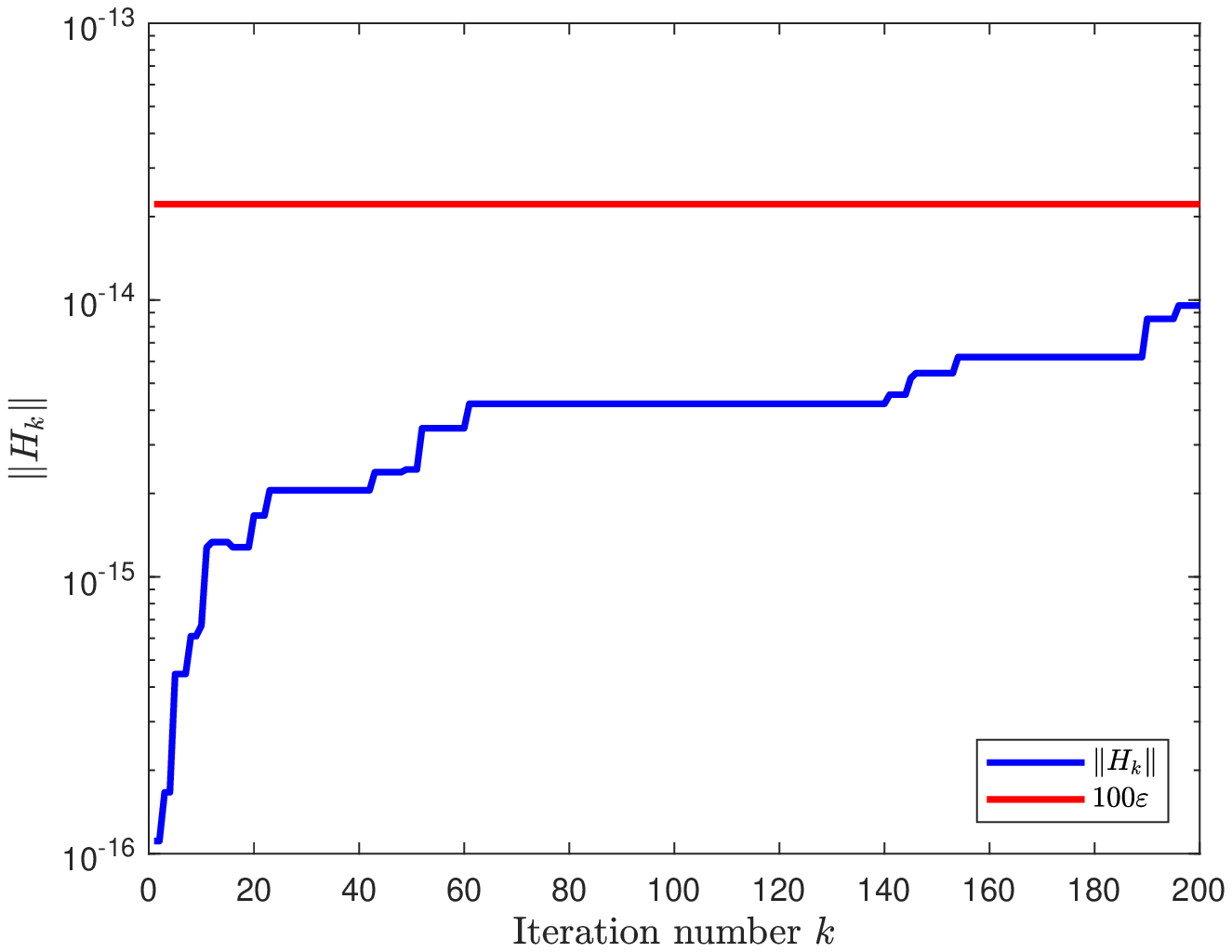}}
		\centerline{(d)}
	\end{minipage}
	\caption{ $\|H_{k}\|$ and its upper bound: (a) {\sf \{$A_{c}$,$L_{s}$\}}; (b) {\sf \{rdb2048,dw2048\}};
		(c) {\sf \{$L_{m}$,ex31\}}; (d) {\sf \{rdb5000,$L_{1}$\}}.}
	\label{fig1}
\end{figure}

Figure \ref{fig1} depicts the the variation of $\|H_{k}\|=\|I_{k}-B_{k}^{T}B_{k}-P\widehat{B}_{k}^{T}\widehat{B}_{k}P\|$ and its upper bound in \eqref{3.13} as the iteration number $k$ increases from $1$ to $200$. Notice \eqref{3.18}. We use $100\epsilon$ as the upper bound of $\|H_{k}\|$. From the four examples, we find that as the matrix dimensions become bigger, $\|H_{k}\|$ grows very slightly as the iteration number $k$ increases,  due to that $\|H_{k}\|$ are dependent on the dimensions of matrices $A$ and $L$.

\begin{figure}[htp]
	\begin{minipage}{0.48\linewidth}
		\centerline{\includegraphics[width=5.5cm,height=3.5cm]{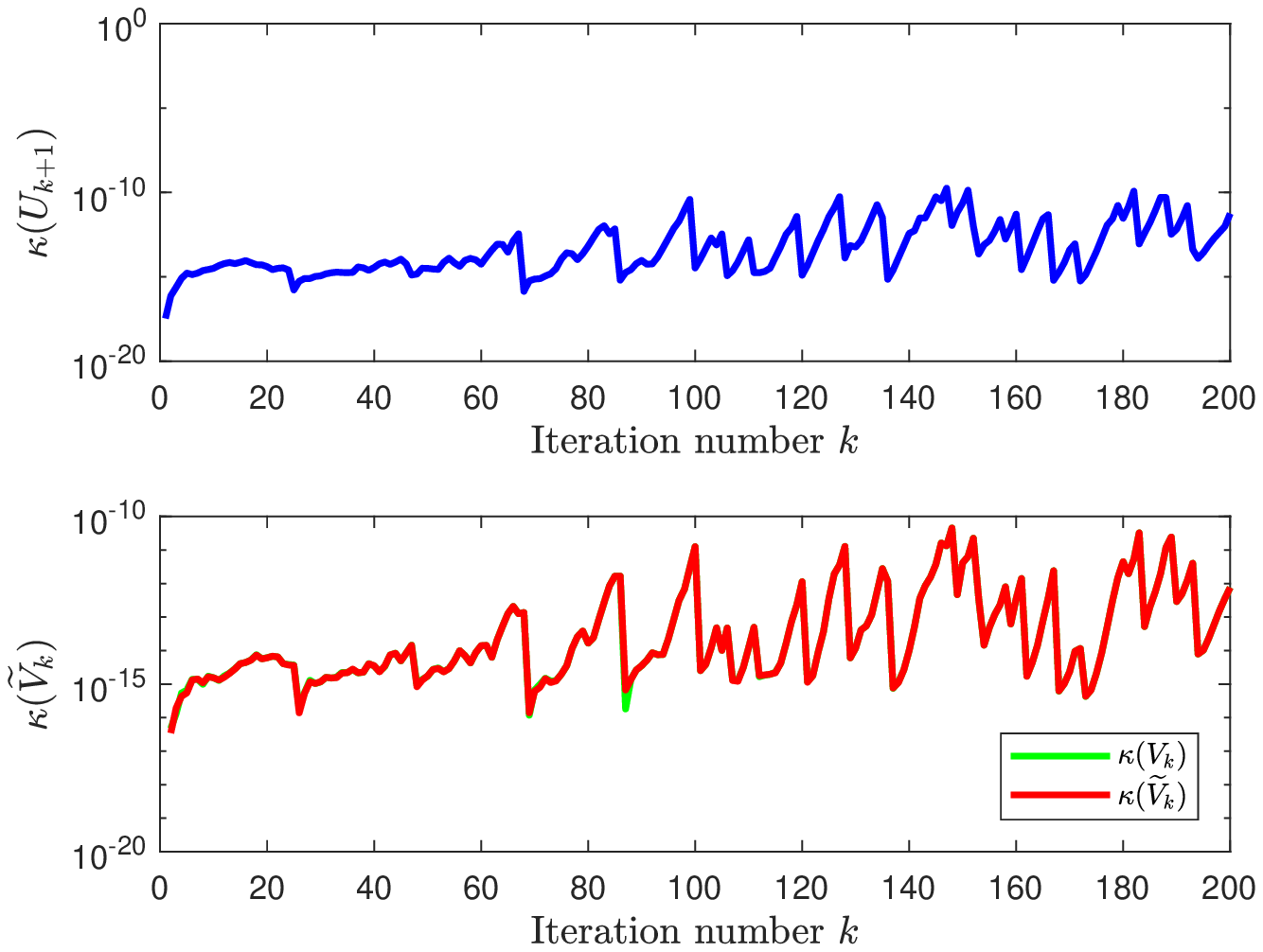}}
		\centerline{(a)}
	\end{minipage}
	\hfill
	\begin{minipage}{0.48\linewidth}
		\centerline{\includegraphics[width=5.5cm,height=3.5cm]{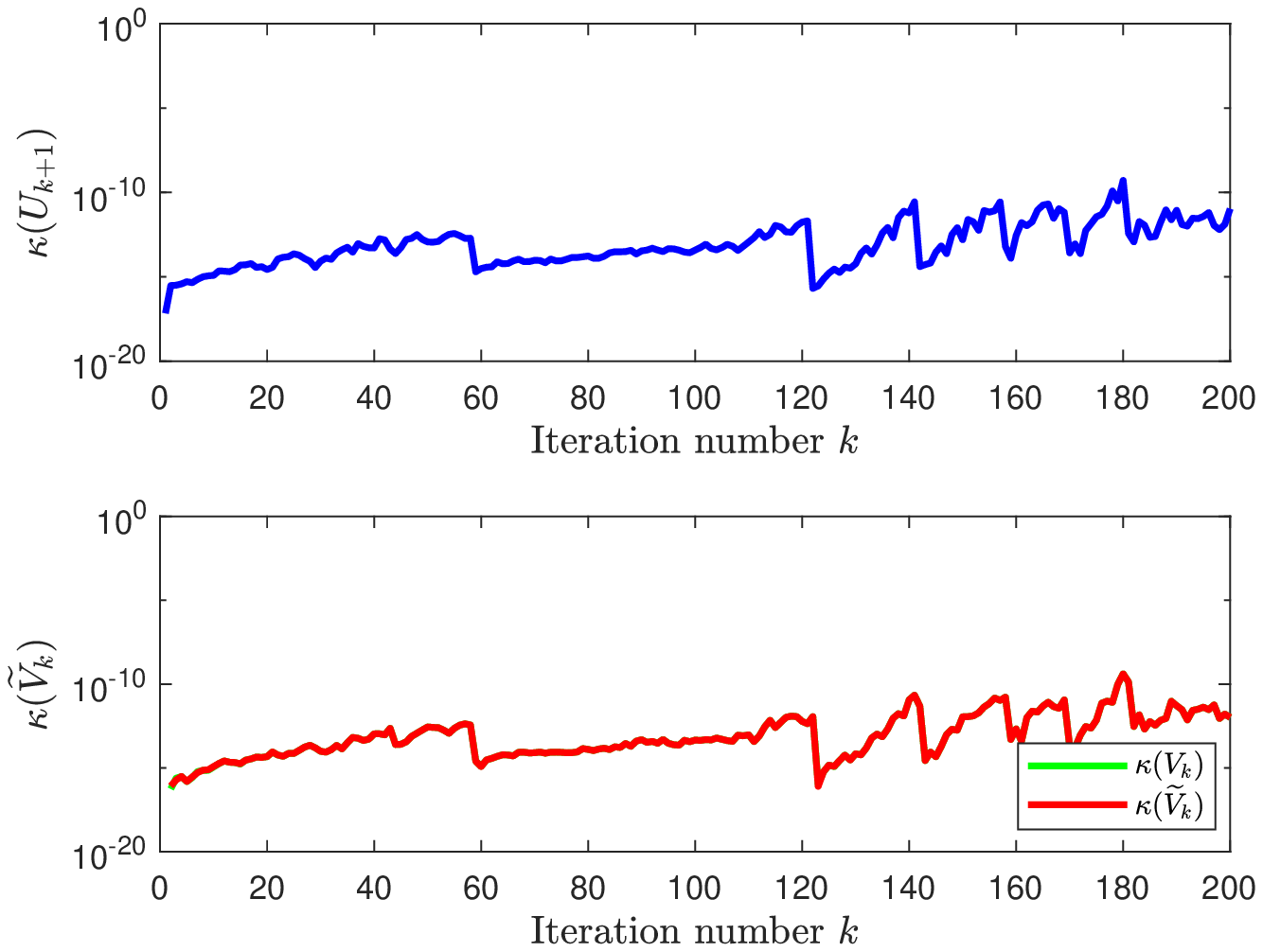}}
		\centerline{(b)}
	\end{minipage}
	\caption{Orthogonality levels of $u_{i}$, $\tilde{v}_{i}$ and $v_{i}$:  (a) {\sf \{$L_{m}$,ex31\}}; (b) {\sf \{rdb5000,$L_{1}$\}}.}
	\label{fig2}
\end{figure}

\begin{figure}[htp]
	\begin{minipage}{0.48\linewidth}
		\centerline{\includegraphics[width=5.5cm,height=3.5cm]{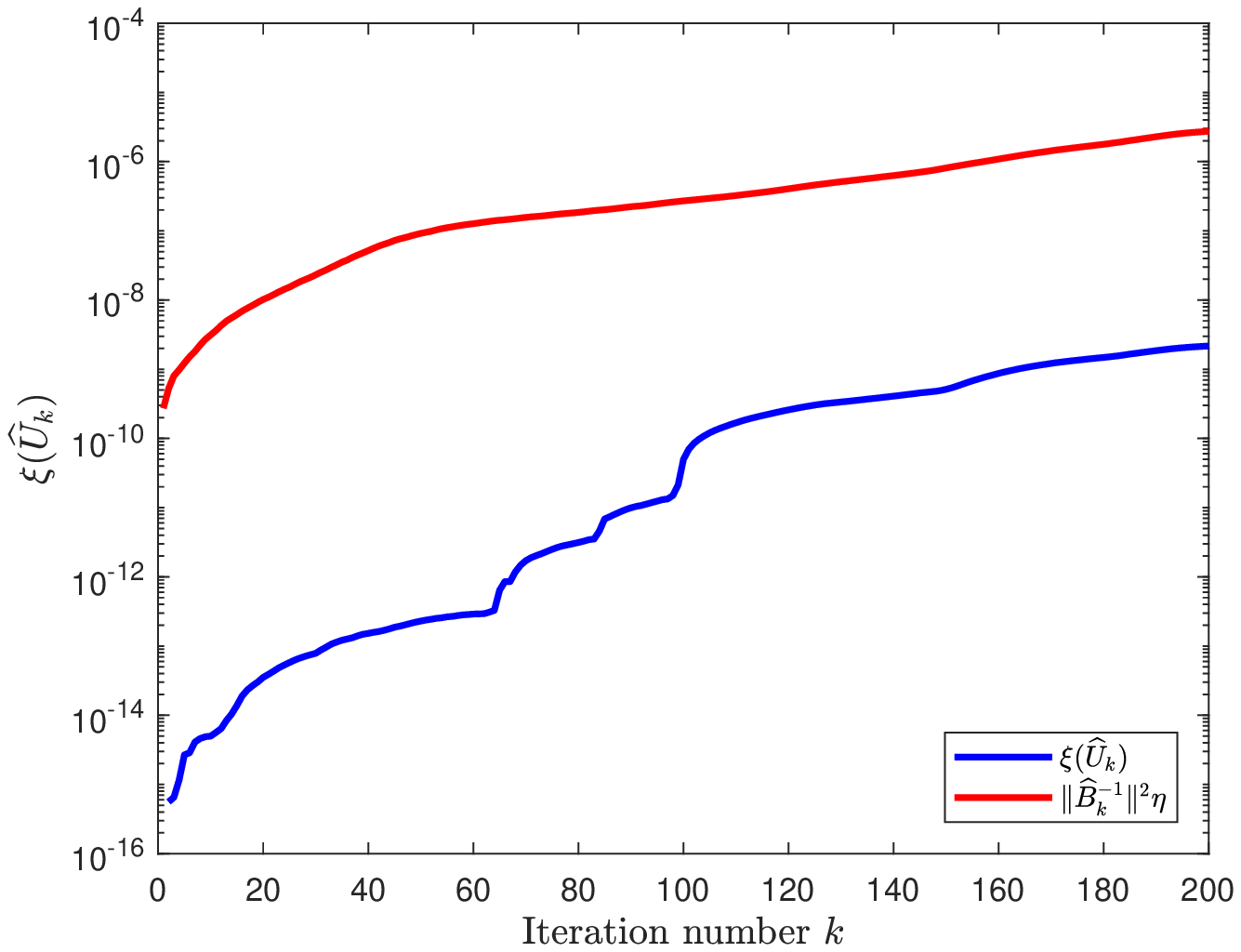}}
		\centerline{(a)}
	\end{minipage}
	\hfill
	\begin{minipage}{0.48\linewidth}
		\centerline{\includegraphics[width=5.5cm,height=3.5cm]{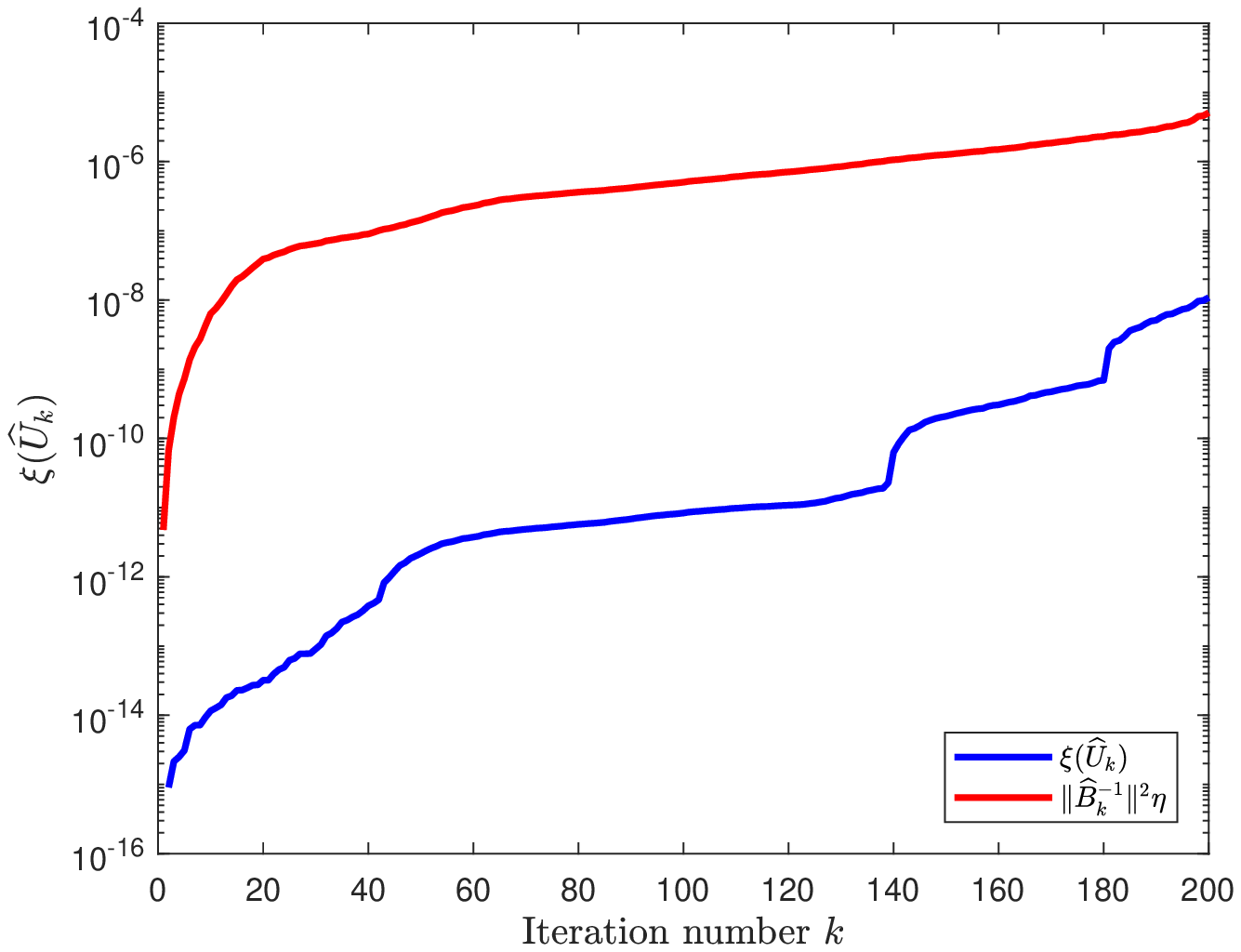}}
		\centerline{(b)}
	\end{minipage}
	\caption{Orthogonality levels of $\hat{u}_{i}$: (a) {\sf \{$L_{m}$,ex31\}}; (b) {\sf \{rdb5000,$L_{1}$\}}.}
	\label{fig3}
\end{figure}

Figure \ref{fig2} and figure \ref{fig3} depict the orthogonality levels of $u_{i}$, $\tilde{v}_{i}$ and $\hat{u}_{i}$ computed by the JBDPRO algorithm. We use matrix pairs {\sf \{$L_{m}$,ex31\}} and {\sf \{rdb5000,$L_{1}$\}} to illustrate the results, the cases of {\sf \{$A_{c}$,$L_{s}$\}} and {\sf \{rdb2048,dw2048\}} are similar and we omit them. The $\eta$-criterion is used and $\eta=\epsilon^{3/4}\approx 10^{-12}$. From the figures we find that in the first few iteration steps, the orthogonality levels of $u_{i}$ and $\tilde{v}_{i}$ grows gradually until they exceed $\eta$, which means the loss of orthogonality. Then, the partial reorthogonalization is applied to $u_{i}$ and $\tilde{v}_{i}$, making the orthogonality levels suddenly jumping down, and then the reorthogonalization is not used in a few later steps until the orthogonality levels exceed $\eta$ again. The algorithm continues in this way and the orthogonality levels of $u_{i}$ and $\tilde{v}_{i}$ fluctuate around $\eta$ as the iteration number $k$ continues increasing. We also depict the orthogonality levels of $v_{i}$, and we can find that the orthogonality levels of $v_{i}$ and $\tilde{v}_{i}$ are almost equal. From figure \ref{fig3}, we find that the orthogonality level of $\hat{u}_{i}$ is mainly affected by the growth of $\|\widehat{B}_{k}^{-1}\|$. If $\|\widehat{B}_{k}^{-1}\|$ does not become too large, then the orthogonality of $\hat{u}_{i}$ will be at a desired level although we do not reorthogonalize any $\hat{u}_{i}$ in the JBDPRO algorithm.

\begin{figure}[htp]
	\begin{minipage}{0.48\linewidth}
		\centerline{\includegraphics[width=5.5cm,height=3.5cm]{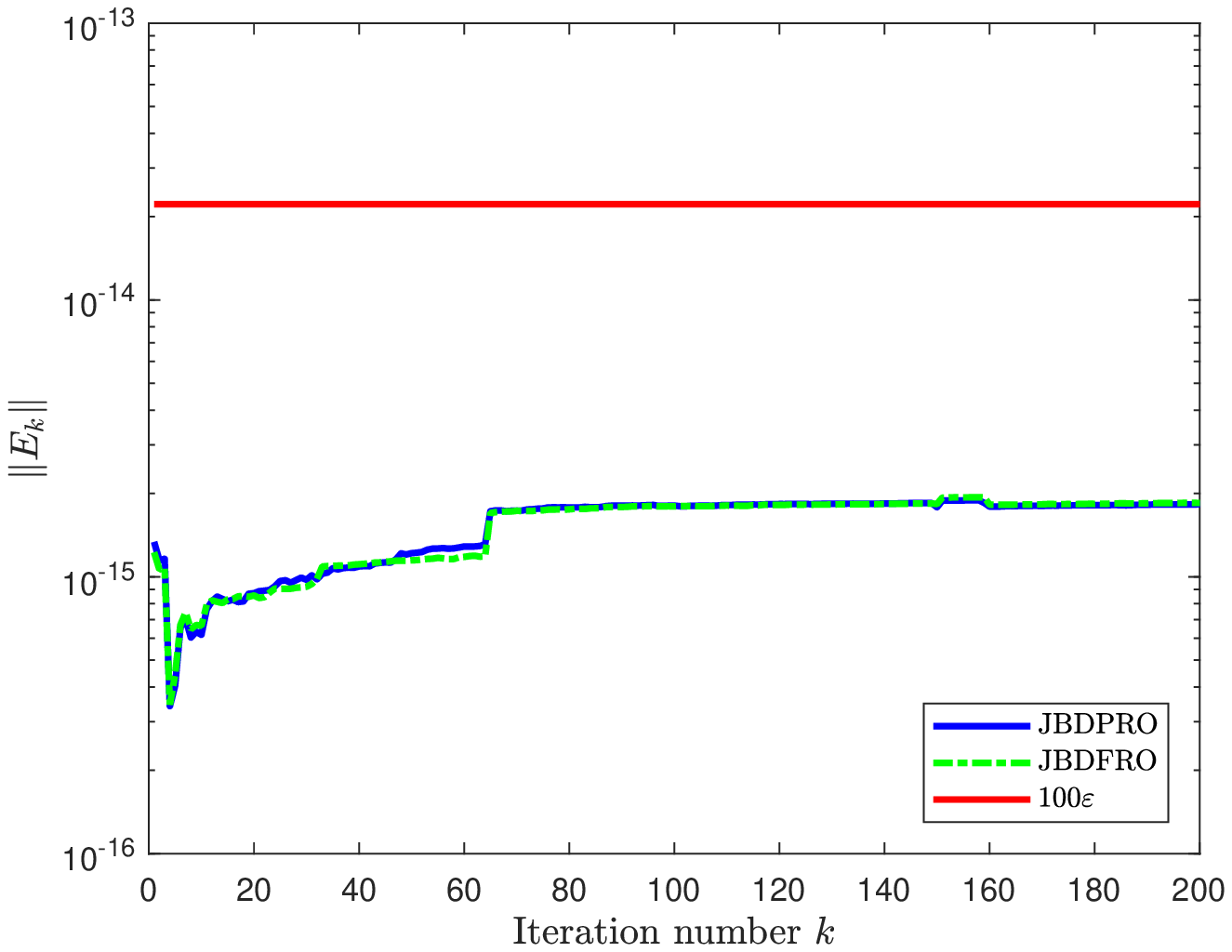}}
		\centerline{(a)}
	\end{minipage}
	\hfill
	\begin{minipage}{0.48\linewidth}
		\centerline{\includegraphics[width=5.5cm,height=3.5cm]{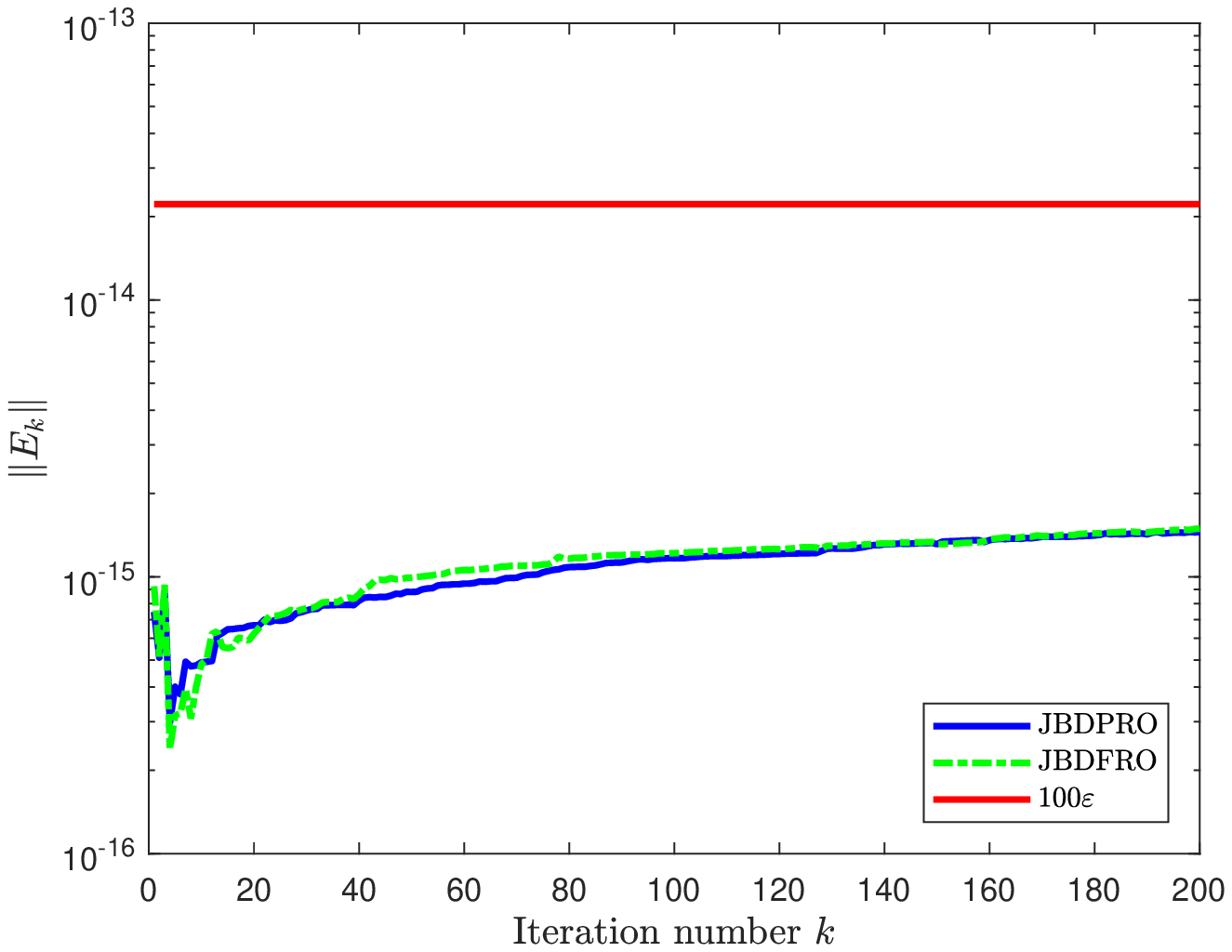}}
		\centerline{(b)}
	\end{minipage}
	
	\vfill
	\begin{minipage}{0.48\linewidth}
		\centerline{\includegraphics[width=5.5cm,height=3.5cm]{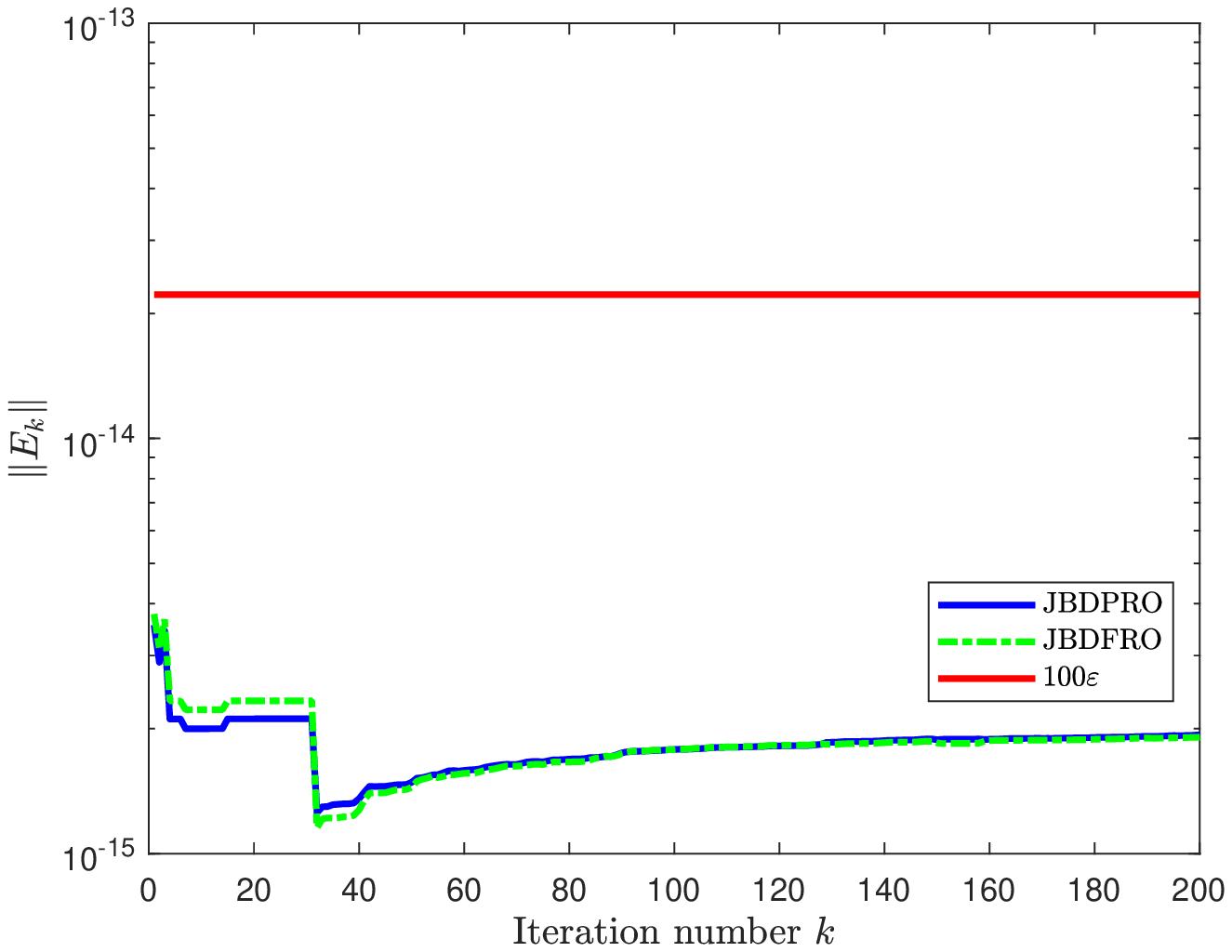}}
		\centerline{(c)}
	\end{minipage}
	\hfill
	\begin{minipage}{0.48\linewidth}
		\centerline{\includegraphics[width=5.5cm,height=3.5cm]{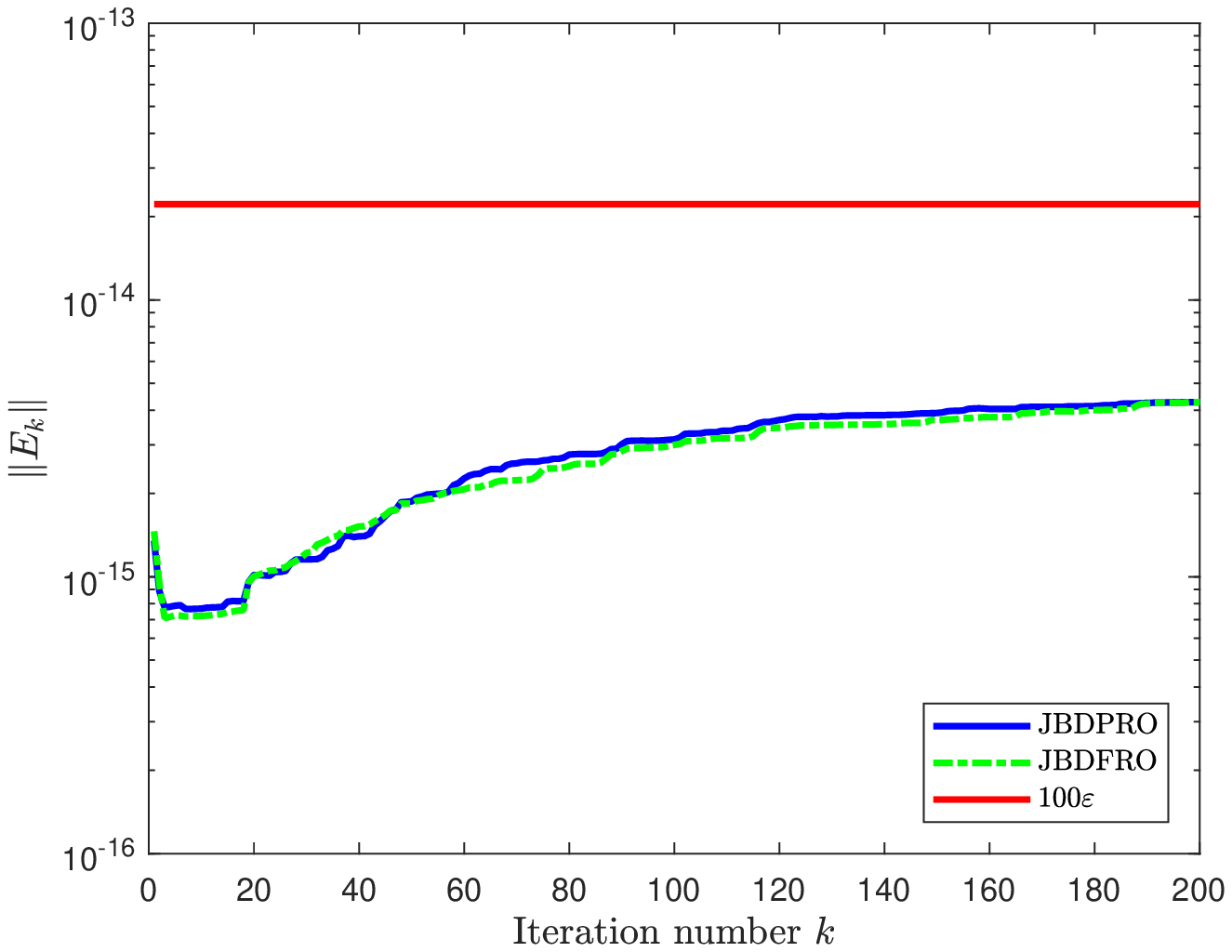}}
		\centerline{(d)}
	\end{minipage}
	\caption{ Comparison of $\|E_{k}\|$ computed by JBDPRO and JBDFRO:(a) {\sf \{$A_{c}$,$L_{s}$\}}; (b) {\sf \{rdb2048,dw2048\}};
		(c) {\sf \{$L_{m}$,ex31\}}; (d) {\sf \{rdb5000,$L_{1}$\}}.}
	\label{fig4}
\end{figure}

\begin{figure}[htp]
	\begin{minipage}{0.48\linewidth}
		\centerline{\includegraphics[width=5.5cm,height=3.5cm]{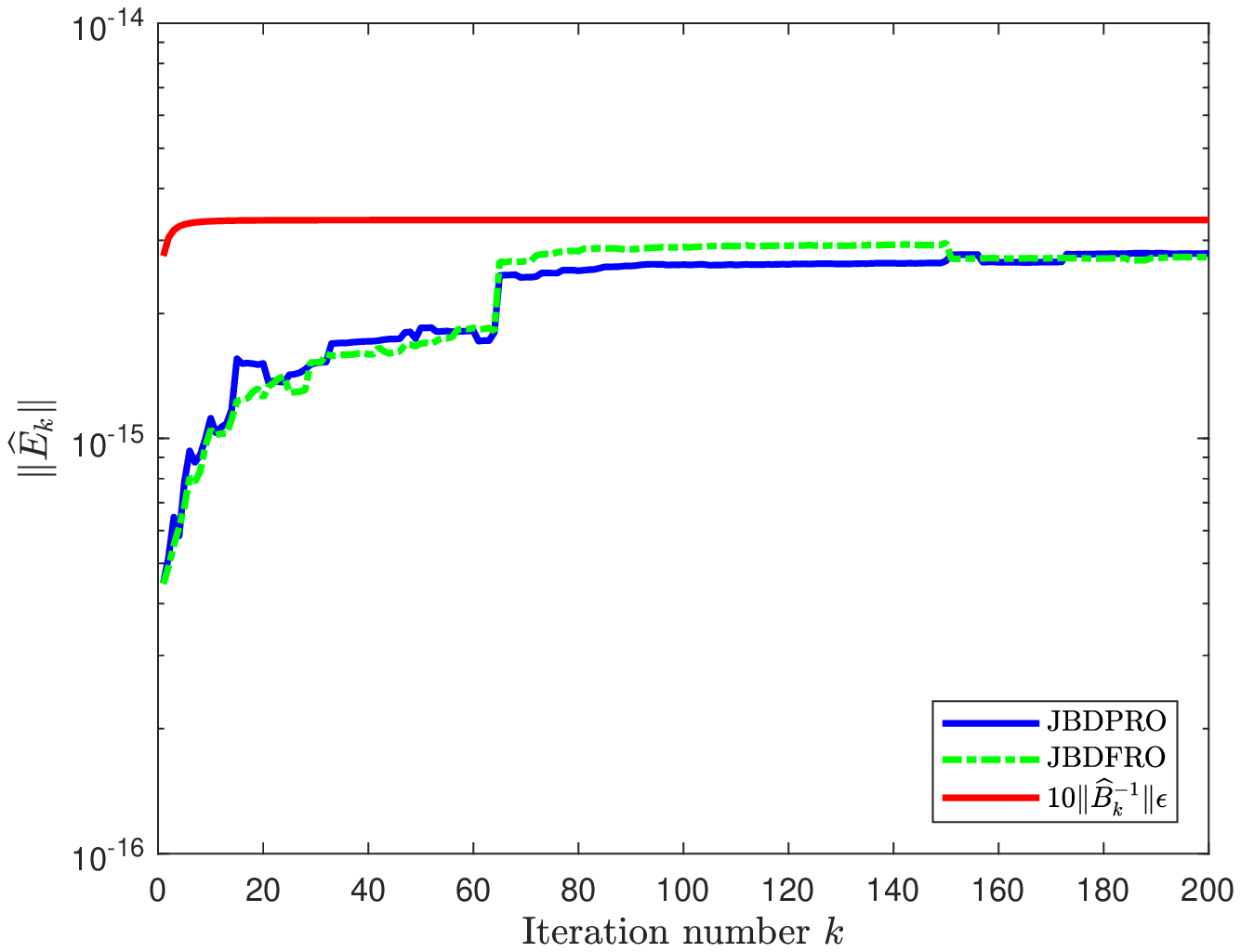}}
		\centerline{(a)}
	\end{minipage}
	\hfill
	\begin{minipage}{0.48\linewidth}
		\centerline{\includegraphics[width=5.5cm,height=3.5cm]{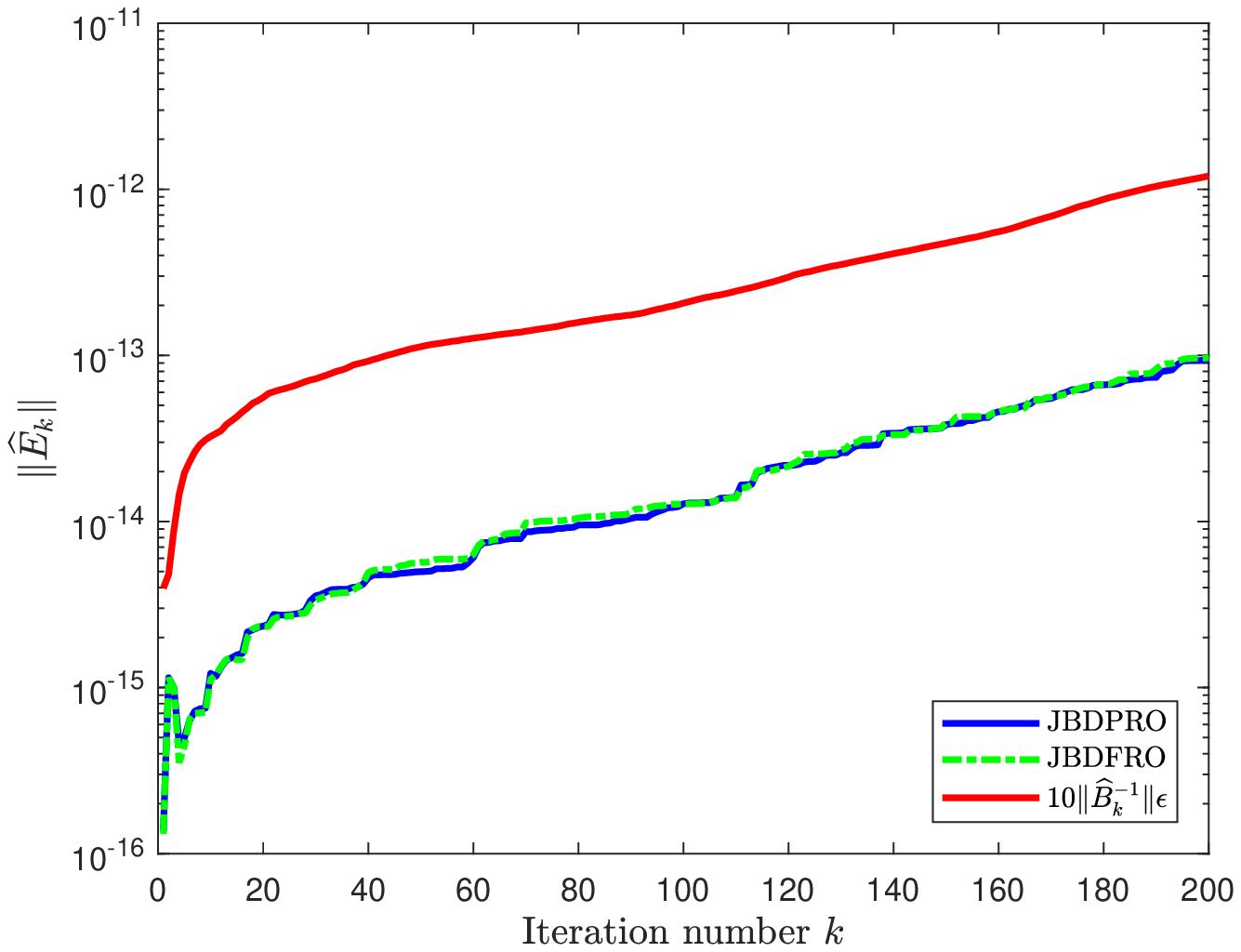}}
		\centerline{(b)}
	\end{minipage}
	
	\vfill
	\begin{minipage}{0.48\linewidth}
		\centerline{\includegraphics[width=5.5cm,height=3.5cm]{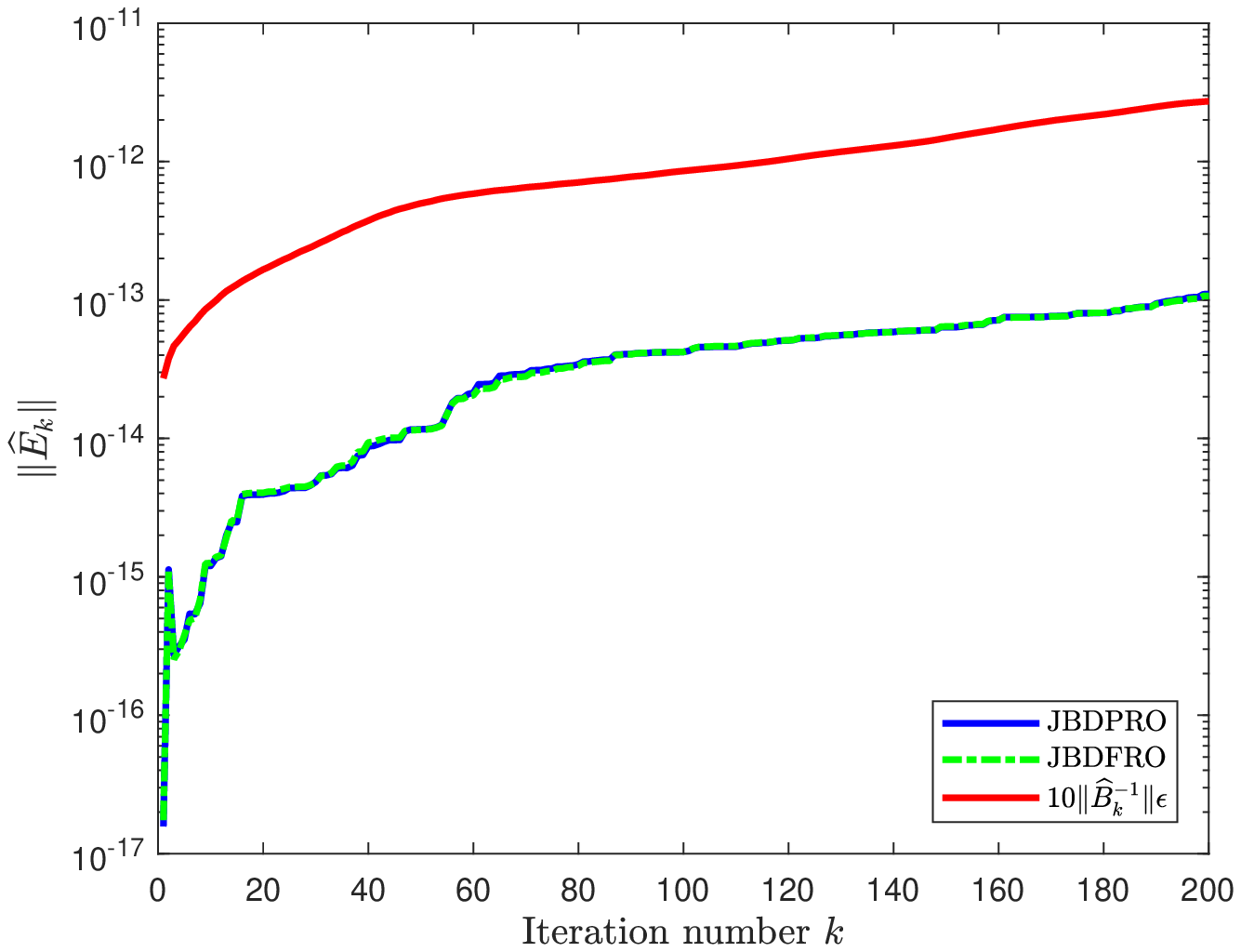}}
		\centerline{(c)}
	\end{minipage}
	\hfill
	\begin{minipage}{0.48\linewidth}
		\centerline{\includegraphics[width=5.5cm,height=3.5cm]{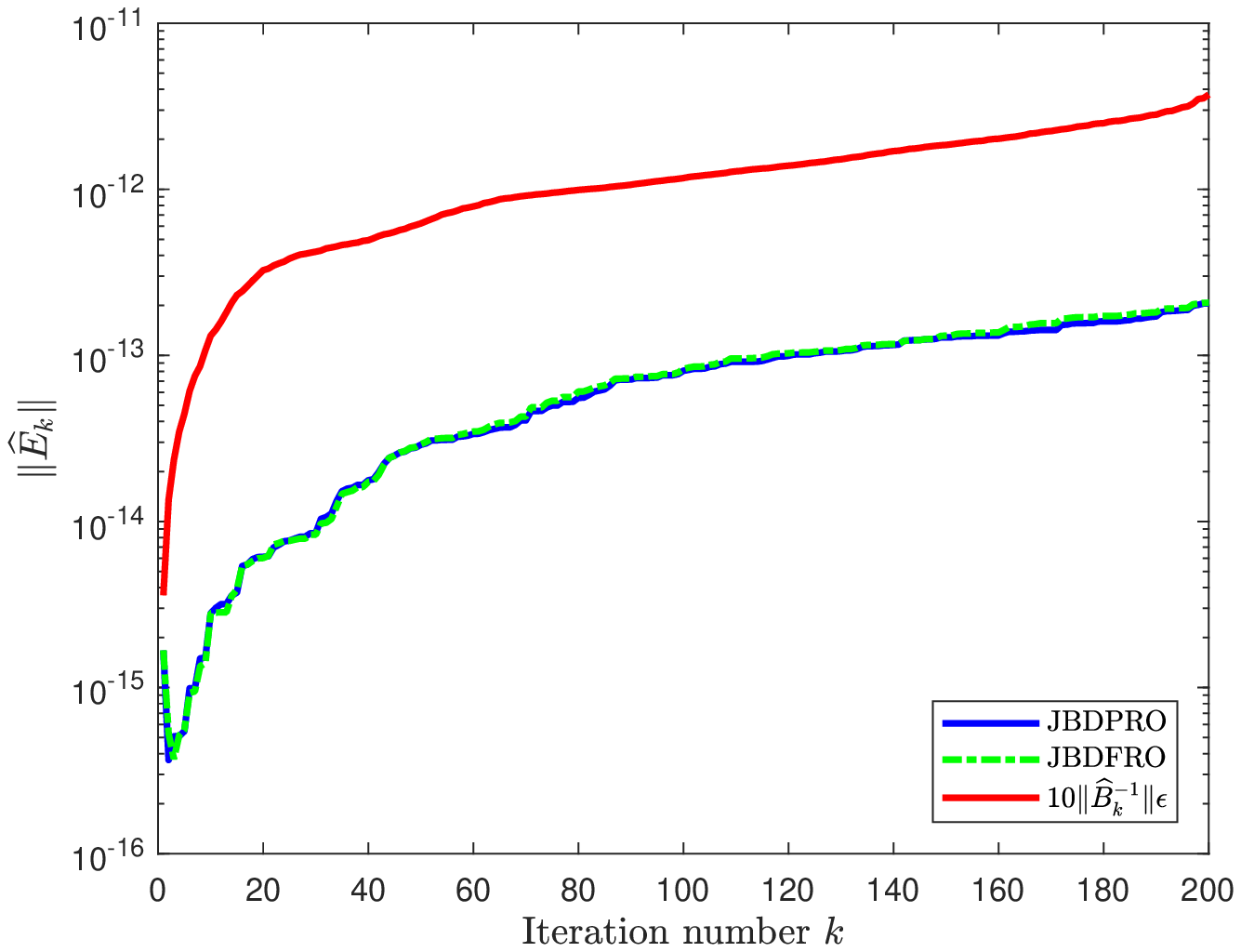}}
		\centerline{(d)}
	\end{minipage}
	\caption{ Comparison of $\|\widehat{E}_{k}\|$ computed by JBDPRO and JBDFRO:(a) {\sf \{$A_{c}$,$L_{s}$\}}; (b) {\sf \{rdb2048,dw2048\}};
		(c) {\sf \{$L_{m}$,ex31\}}; (d) {\sf \{rdb5000,$L_{1}$\}}.}
	\label{fig5}
\end{figure}

Now we compare the JBDPRO algorithm with the joint bidiagonalization process with full reorthognalization(JBDFRO). The JBDFRO algorithm uses the full reorthogonalization strategy for $u_{i}$, $\tilde{v}_{i}$ and $\hat{u}_{i}$ at each step, and the computed $U_{k+1}$, $\widetilde{V}_{k}$ and $\widehat{U}_{k}$ are orthogonalized to machine precision $\epsilon$. Figure \ref{fig5} and figure \ref{fig6} depict the variation of $\|E_{k}\|$ and $\|\widehat{E}_{k}\|$ computed by JBDPRO and JBDFRO, respectively. From these figures, we can find that both $\|E_{k}\|$ and $\|\widehat{E}_{k}\|$ computed by JBDPRO and JBDFRO are almost the same. For the four examples, the quantity $\|E_{k}\|$ does not deviate far from $\epsilon$ and $100\epsilon$ is an upper bound, while $\|\widehat{E}_{k}\|$ grows slightly and the growth speed is mainly affected by that of $\|\widehat{B}_{k}^{-1}\|$.

We show the convergence of Ritz values computed from the SVD of $B_{k}$ or $\widehat{B}_{k}$ computed by the JBD process, with and without the semiorthogonalization strategy, respectively. The matrix pair $\{A,L\}$ is constructed as follows. Let $m=n=p=800$. First, construct a vector $c$ such that $c(1)=0.90$,  $c(2)=c(3)=0.86$, $c(4)=0.82$,  $c(5)=0.78$, $c(796)=0.22$, $c(797)=0.20$, $c(798)=c(799)=0.15$, $c(800)=0.10$ and $\texttt{c(6:795)=linspace(0.80,0.30,790)}$ generated by the MATLAB built-in function \texttt{linspace()}. Then let $s = ((1-c_{1}^{2})^{1/2}, \dots, (1-c_{n}^{2})^{1/2})$. Let $C=diag(c)$, $S=diag(s)$ and $\texttt{D = gallery(`orthog',n,2)}$, which means that $D$ is a symmetric orthogonal matrix. Finally let $A=CD$ and $L=SD$. By the construction, we know that the $i$-th generalized singular value of $\{A,L\}$ is $\{c_{i},s_{i}\}$, and the multiplicities of the generated singular values $\{0.86, \sqrt{1-0.86^{2}}\}$ and $\{0.15, \sqrt{1-0.15^{2}}\}$ are 2.

\begin{figure}[htp]
	\begin{minipage}{0.48\linewidth}
		\centerline{\includegraphics[width=5.5cm,height=3.5cm]{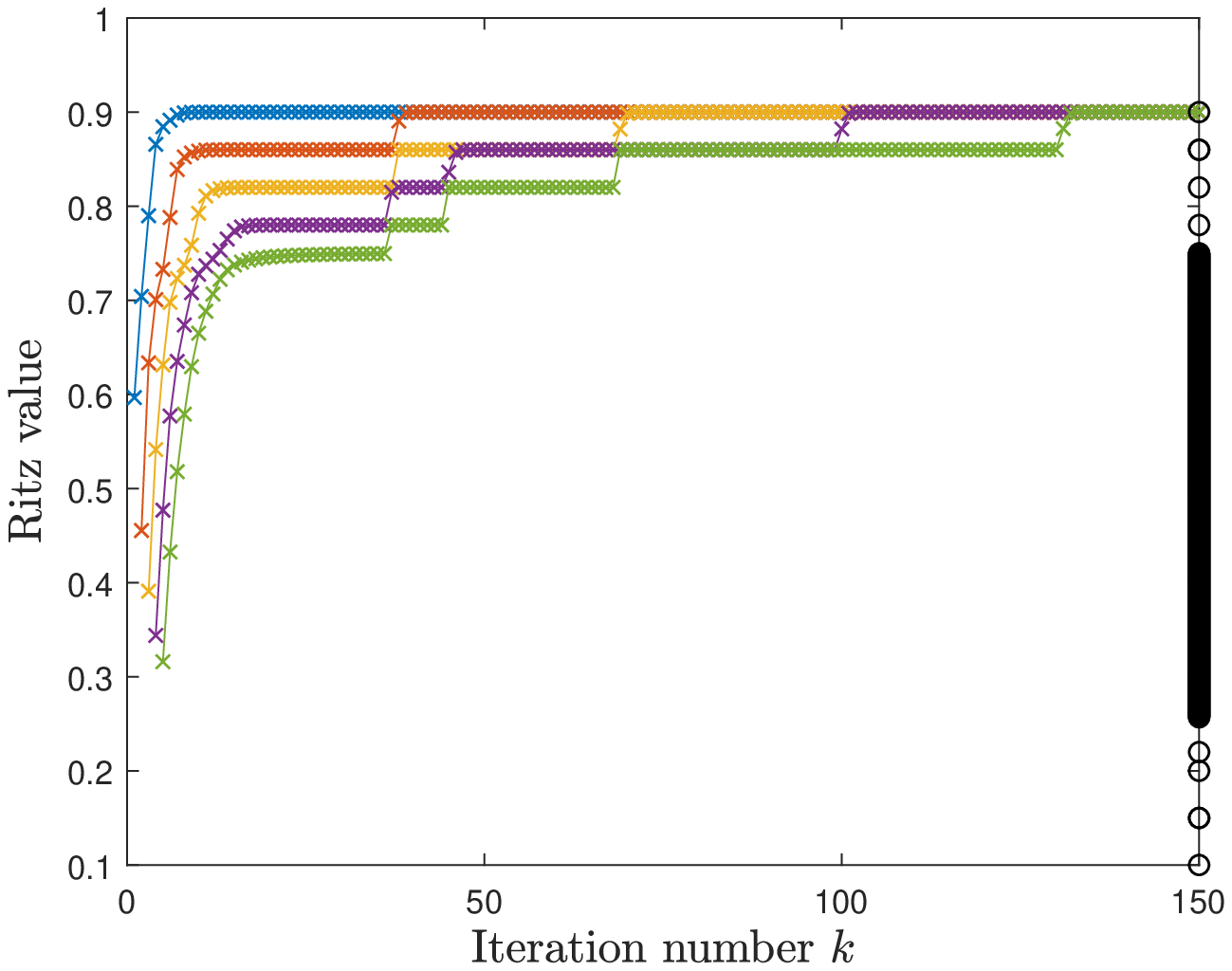}}
		\centerline{(a)}
	\end{minipage}
	\hfill
	\begin{minipage}{0.48\linewidth}
		\centerline{\includegraphics[width=5.5cm,height=3.5cm]{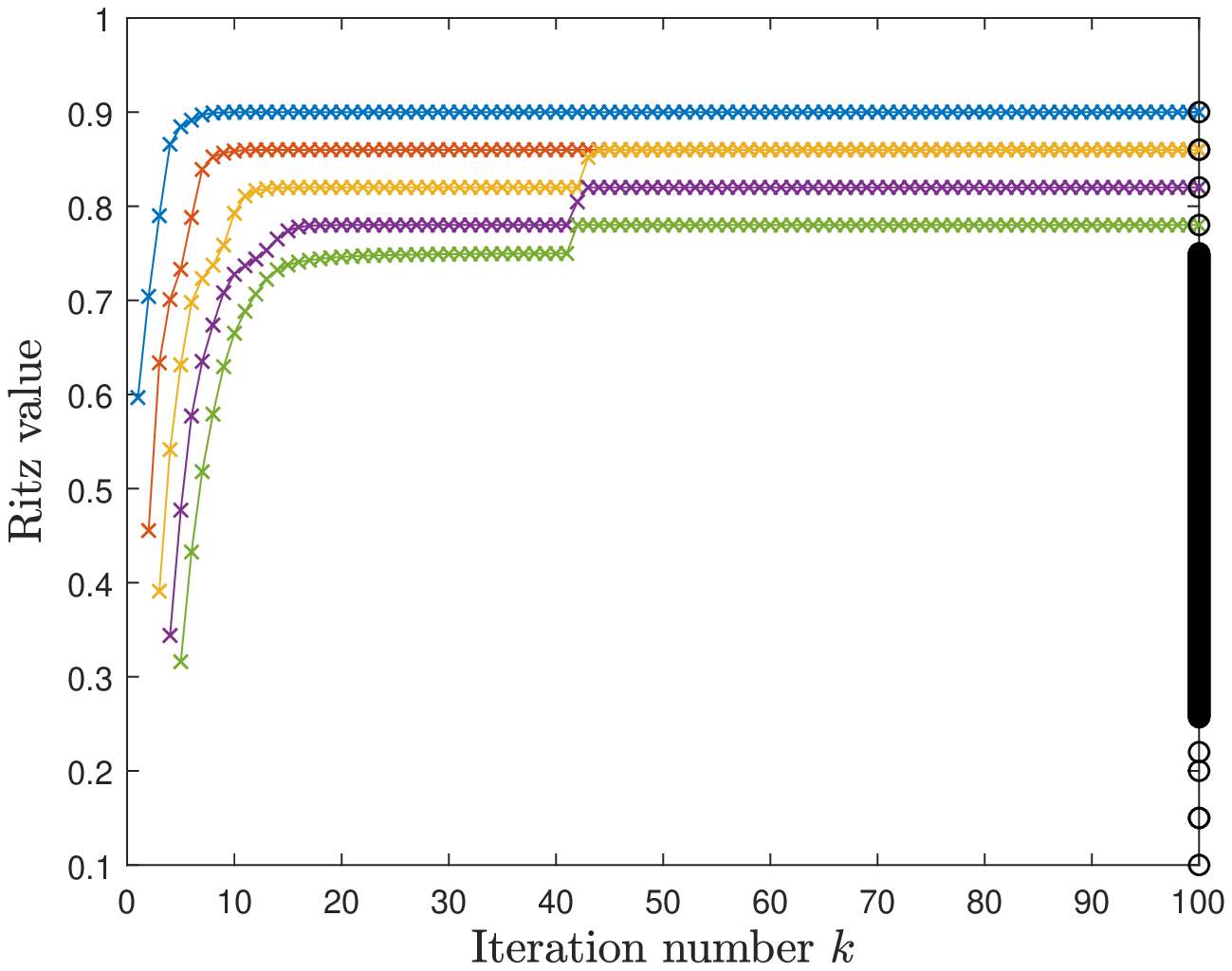}}
		\centerline{(b)}
	\end{minipage}
	
	\vfill
	\begin{minipage}{0.48\linewidth}
		\centerline{\includegraphics[width=5.5cm,height=3.5cm]{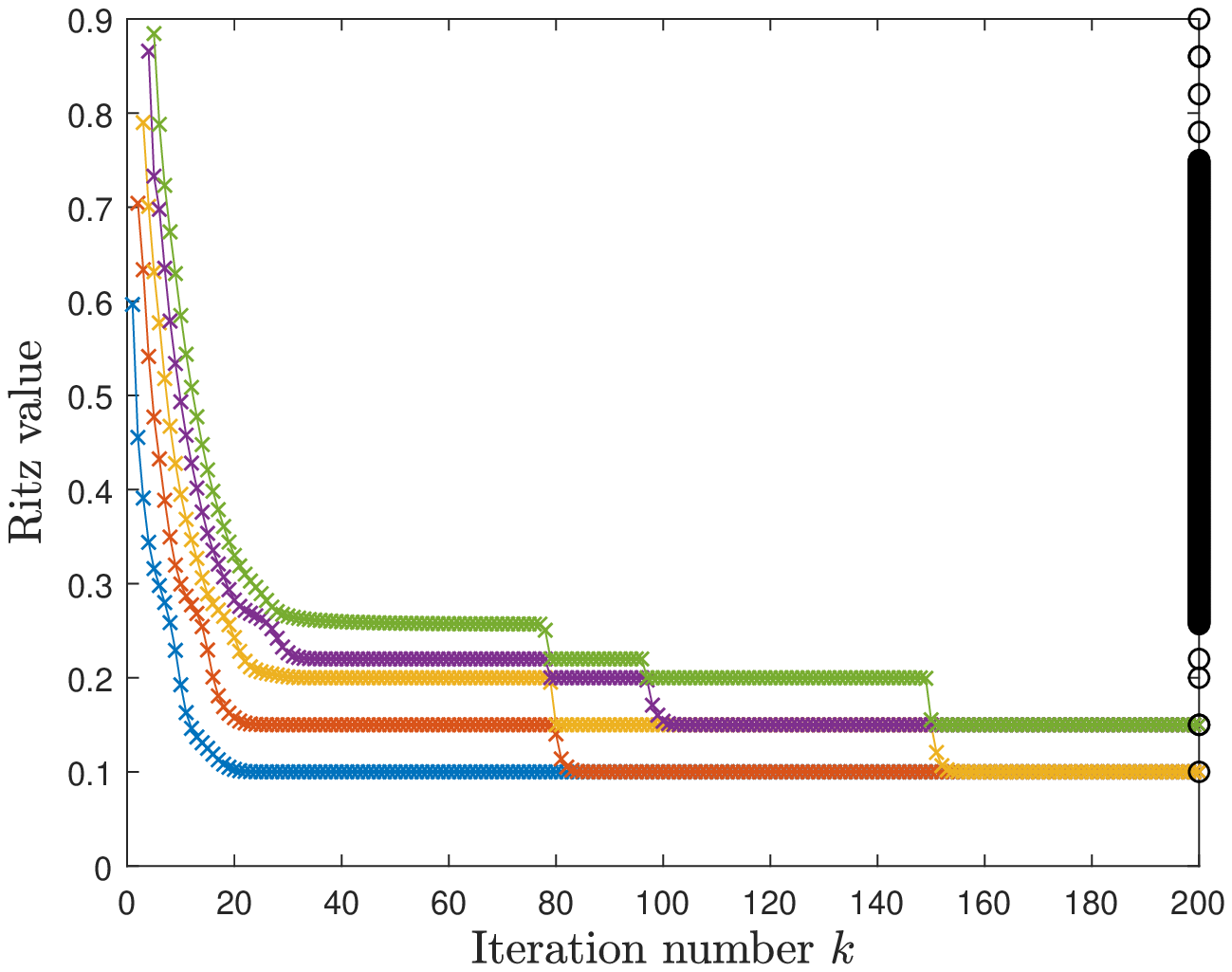}}
		\centerline{(c)}
	\end{minipage}
	\hfill
	\begin{minipage}{0.48\linewidth}
		\centerline{\includegraphics[width=5.5cm,height=3.5cm]{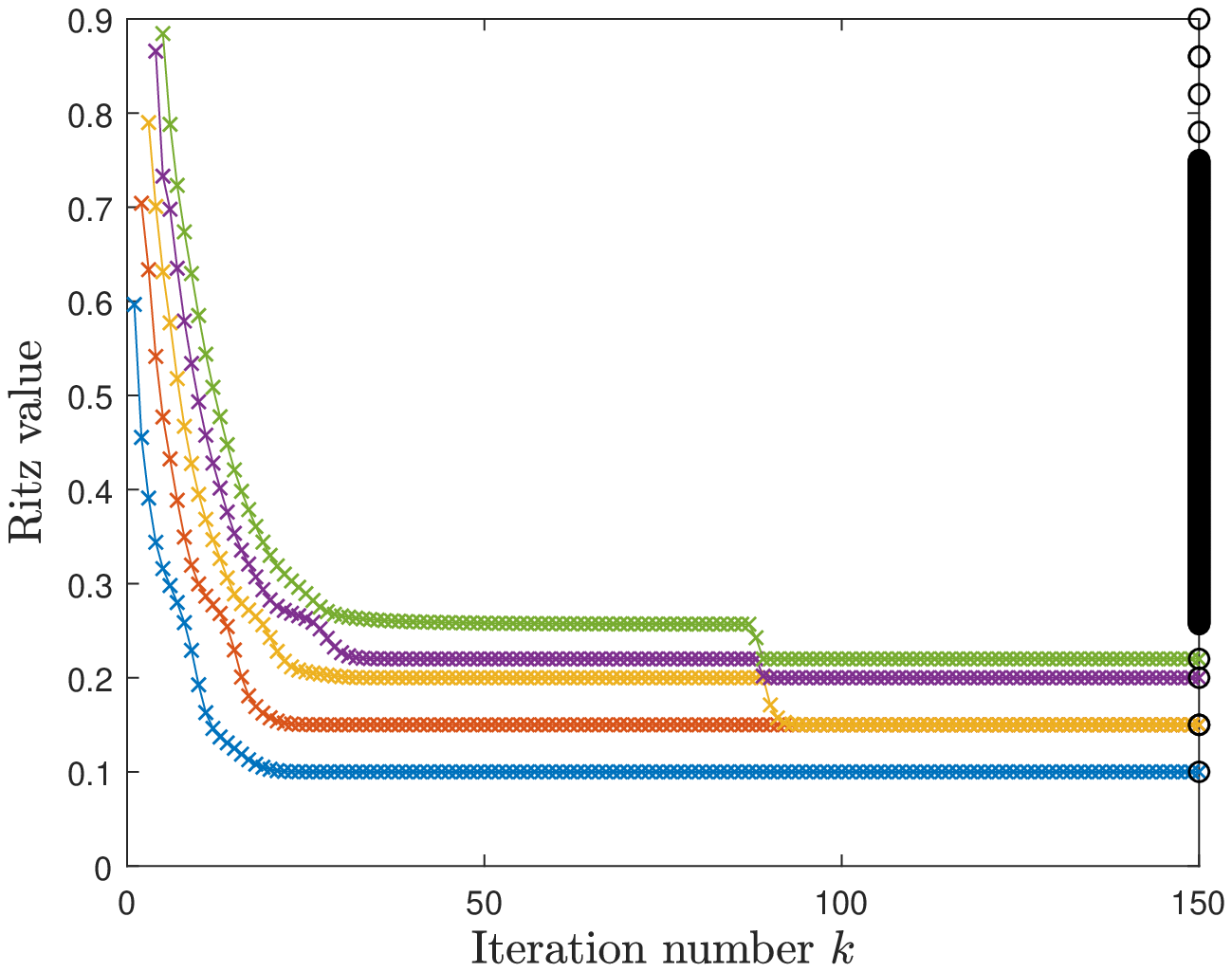}}
		\centerline{(d)}
	\end{minipage}
	\caption{Convergence of Ritz values from the SVD of $B_{k}$: (a) {the first five largest Ritz values, computed by JBD}; (b) {the first five largest Ritz values, computed by JBDPRO}; (c) {the first five smallest Ritz values, computed by JBD}; (d) {the first five smallest Ritz values, computed by JBDPRO}.}
	\label{fig6}
\end{figure}

\begin{figure}[htp]
	\begin{minipage}{0.48\linewidth}
		\centerline{\includegraphics[width=5.5cm,height=3.5cm]{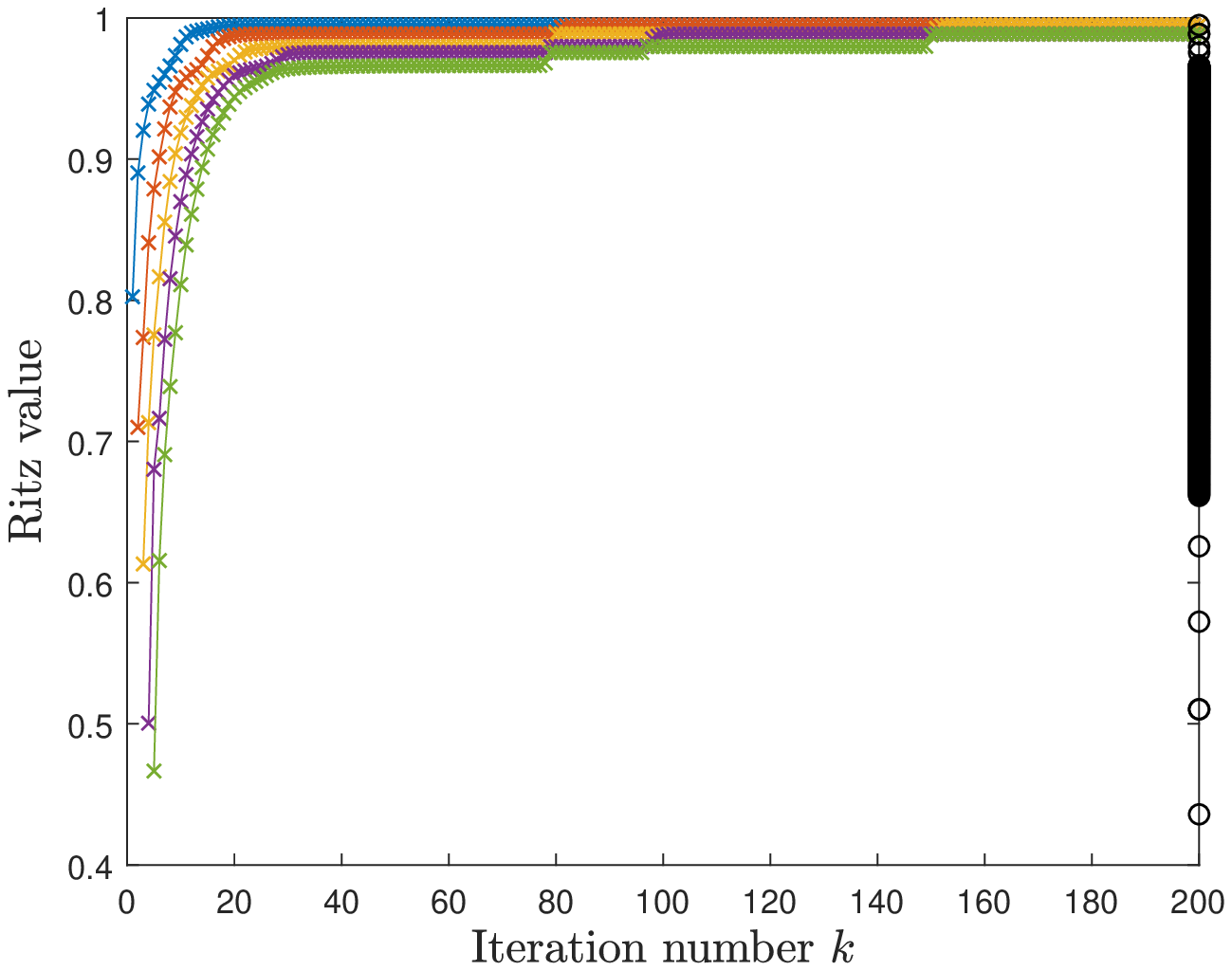}}
		\centerline{(a)}
	\end{minipage}
	\hfill
	\begin{minipage}{0.48\linewidth}
		\centerline{\includegraphics[width=5.5cm,height=3.5cm]{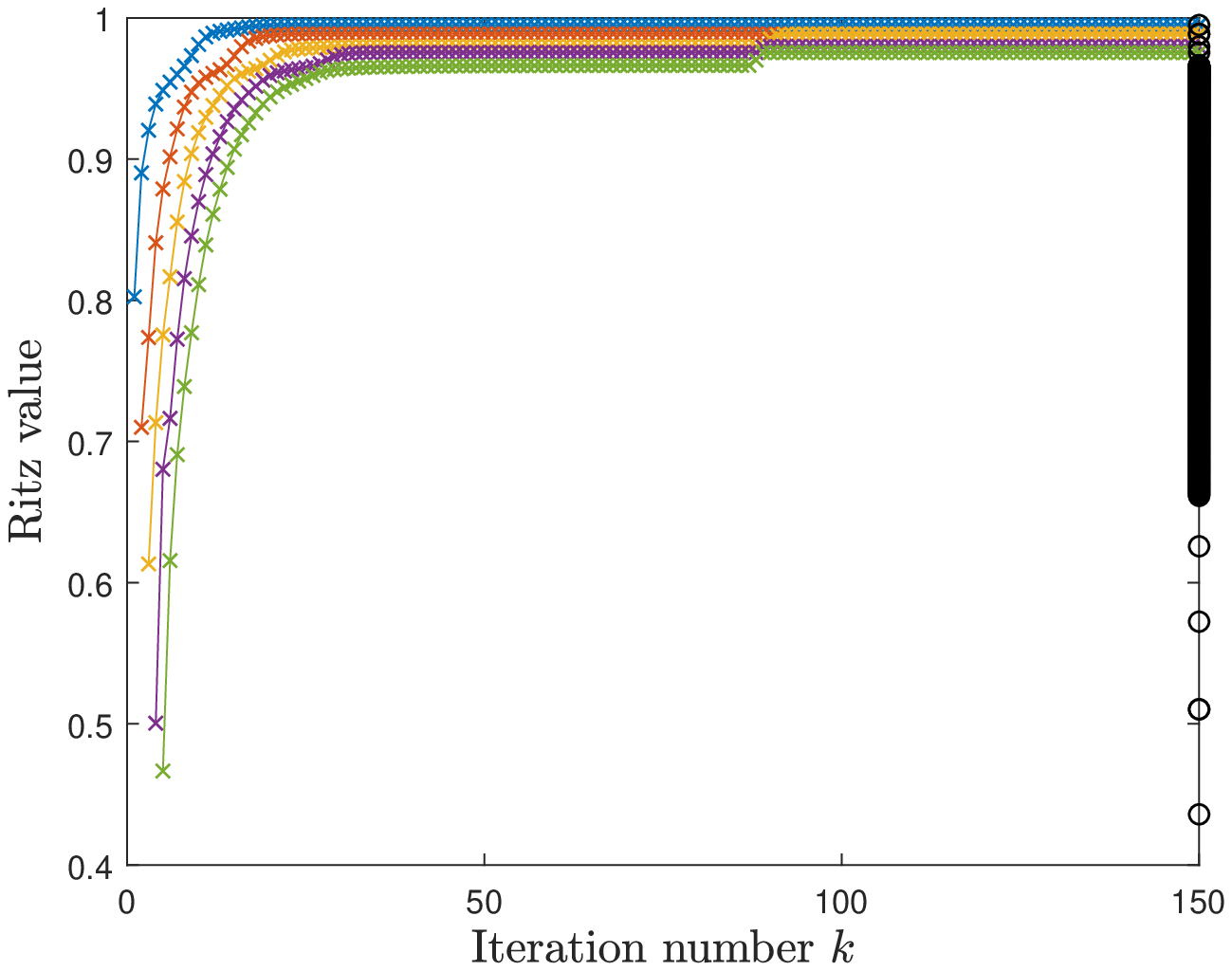}}
		\centerline{(b)}
	\end{minipage}
	
	\vfill
	\begin{minipage}{0.48\linewidth}
		\centerline{\includegraphics[width=5.5cm,height=3.5cm]{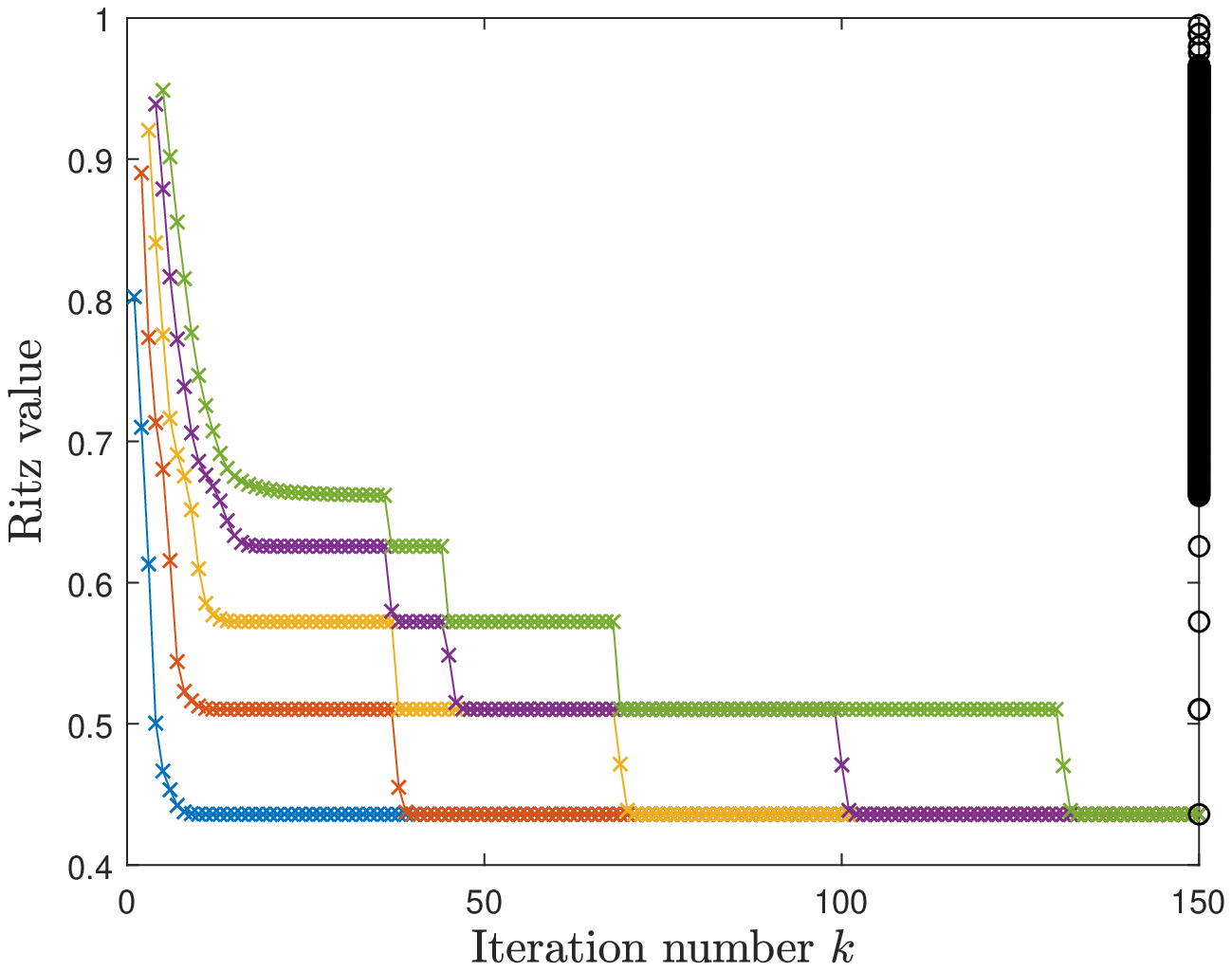}}
		\centerline{(c)}
	\end{minipage}
	\hfill
	\begin{minipage}{0.48\linewidth}
		\centerline{\includegraphics[width=5.5cm,height=3.5cm]{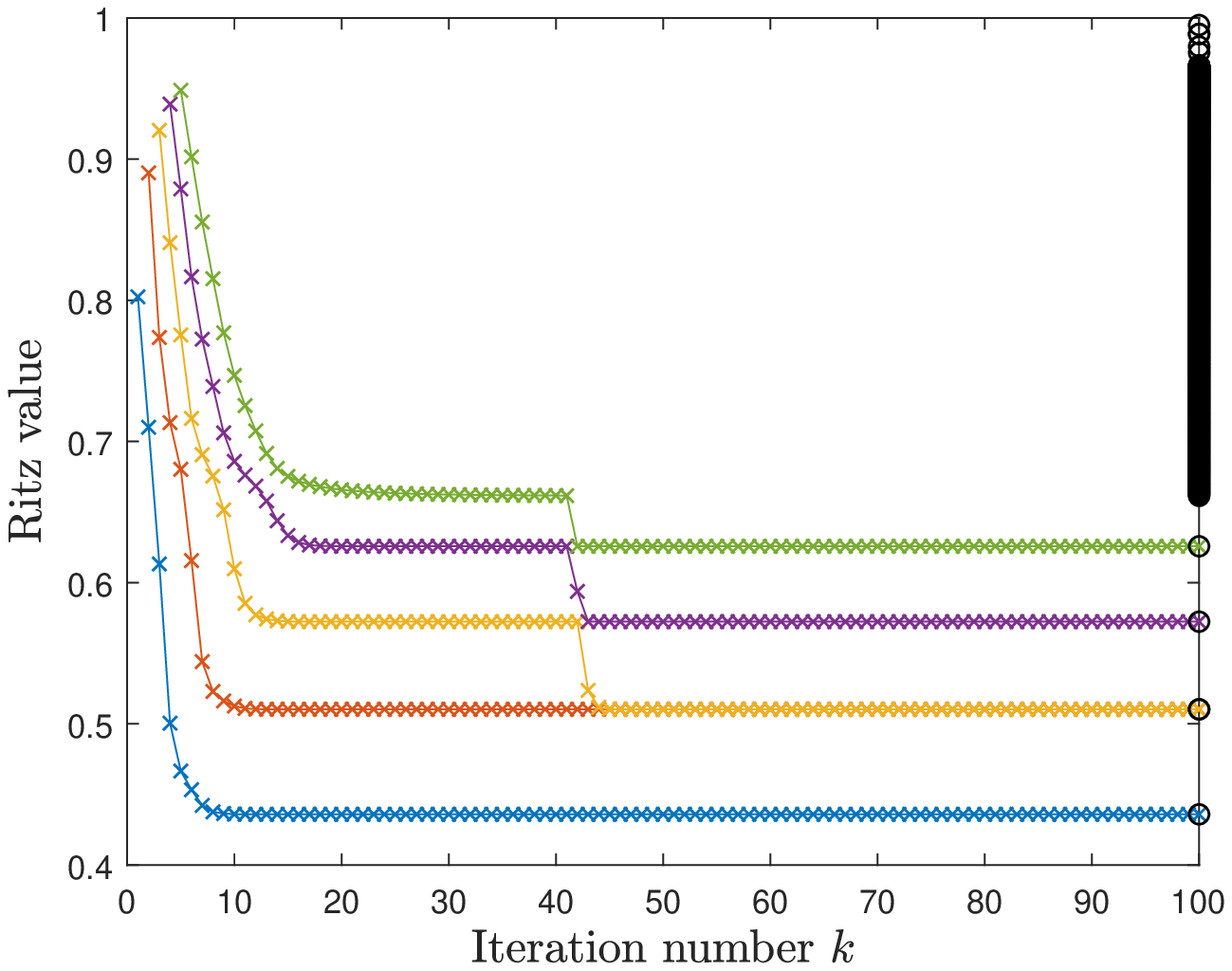}}
		\centerline{(d)}
	\end{minipage}
	\caption{ Convergence of Ritz values from the SVD of $\widehat{B}_{k}$: (a) {the first five largest Ritz values, computed by JBD}; (b) {the first five largest Ritz values, computed by JBDPRO}; (c) {the first five smallest Ritz values,computed by JBD}; (d) {the first five smallest Ritz values, computed by JBDPRO}.}
	\label{fig7}
\end{figure}

Figure \ref{fig6} depicts the convergence of the first five largest and smallest Ritz values from the SVD of $B_{k}$ computed by the JBD and JBDPRO algorithms, respectively. The right horizontal line indicates the values of $c_{i}$ for $i=1,\dots, 800$. In the left panel, which shows the convergence behavior without reorthogonalization, we see the phenomenon that some of the converged Ritz values suddenly ``jump" to become ``ghosts" and then converge to the next larger or smaller singular values after a few iterations, which results to many unwanted spurious copies of generalized singular values and makes it difficult to determine whether these spurious copies are real multiple generalized singular values. In the right panel, where $B_{k}$ is computed by the JBDPRO algorithm, the convergence behavior is much simpler and it is similar to the ideal case in exact arithmetic. It can be found from subfigures (b) and (d), that a simple generalized singular value can be approximated by Ritz values with no ghosts appearing, while a multiple generalized singular value can be approximated one by one by the Ritz values. 

Figure \ref{fig7} depicts the convergence of the first five largest and smallest Ritz values from the SVD of $\widehat{B}_{k}$, where the right horizontal line indicates the value of $s_{i}$ for $i=1,\dots, 800$. The convergence behavior of the Ritz values from the SVD of $\widehat{B}_{k}$ is very similar to that from the SVD of $B_{k}$. From subfigures (a) and (c), which show the convergence of Ritz values without reorthogonalization, we find the ``ghosts" phenomenon that some converged Ritz values suddenly ``jump" and then converge to the next larger or smaller singular values after a few iterations. In subfigures (b) and (d), where $\widehat{B}_{k}$ is computed by the JBDPRO algorithm, the spurious copies are prohibited from appearing, and the multiplicities of the generalized singular values can be determined correctly from the convergence of Ritz values. 

\begin{table}[]
	\centering
	\caption{Running time comparison(measured in seconds)}
	\begin{tabular}{|l|l|l|l|l|l|}
		\hline 
		$A$   		    & $L$  	 	    & JBD      & JBDPRO   & JBDFRO  & ratio	\\  \hline  
		$A_{c}$		    & $L_{s}$       & 0.2528   & 0.2639   & 0.4801  & 54.98\% 	\\  
		{\sf rdb2048}   & {\sf dw2048}  & 2.0048   & 2.2476   & 2.7790  & 80.88\%	\\  
		$L_{m}$         & {\sf ex31}    & 6.8089   & 7.0218   & 9.4157  & 74.58\%	\\
		{\sf rdb5000}	& $L_{1}$       & 10.4968  & 10.7883  & 14.2574 & 75.67\%	\\ \hline 
	\end{tabular}
	\label{tab2}
\end{table}

Finally, we compare the efficiency of the JBDPRO and JBDFRO. Table \ref{tab2} shows the running time of 200-step JBD, JBDPRO and JBDFRO for the four test examples. We also compute the ratio of the running times of JBDPRO and JBDFRO. For each case, we run the algorithms 10 times and take the average over all 10 running times. From the table, we find that the running time of JBDPRO is only about 70\%--80\% of that of JBDFRO. Therefore, the JBDPRO is more efficient than JBDFRO while can avoid ``ghosts" from appearing.

\section{Conclusion}\label{sec6}
We have proposed a semiorthogonalization strategy for the JBD process to maintain some level of orthongonality of the Lanczos vectors. Our rounding error analysis establishes connections between the JBD process with the semiorthonalization strategy and the Lanczos bidiagonalization process. We have proved that if the Lanczos vectors are kept semiorthogonal, the computed $\widehat{B}_{k}$ is the Ritz-Galerkin projection of $Q_{L}$ on the subspaces $span(\widehat{U}_{k})$ and $span(\widehat{V}_{k})$ within error $\delta=O(c_4(m,n,p)\|\widehat{B}_{k}^{-1}\|\epsilon)$. Therefore, the convergence of Ritz values computed from $\widehat{B}_{k}$ will not be affected by rounding errors and the final accuracy of computed quantities is high enough as long as $\|\widehat{B}_{k}^{-1}\|$ does not become too large.

Based on the semiorthogonalization strategy, we have developed the JBDPRO algorithm. The JBDPRO algorithm can keep the Lanczos vectors at a desired level and saves much unnecessary reorthogonalization work compared with the JBDFRO algorithm. Several numerical examples have been used to confirm our theory and algorithm.

\appendix
\section{Appendix: Proofs of Lemma \ref{Lem3.1} and Lemma \ref{Lem3.2}}\label{Apd}

\begin{proof}[Proof of Lemma \ref{Lem3.1}]
	We prove \eqref{3.11} by mathematical induction. For the base case $i=1$,  from \eqref{3.9} and \eqref{3.10} we have
	\begin{align*}
	\hat{\alpha}_{1}Q_{L}^{T}\hat{u}_{1} 
	& = Q_{L}^{T}Q_{L}\hat{v}_{1} -Q_{L}^{T}\hat{f}_{1} \\
	& = (I_{n}-Q_{A}^{T}Q_{A})\hat{v}_{1} -Q_{L}^{T}\hat{f}_{1} \\
	& = \hat{v}_{1}-Q_{A}^{T}(\alpha_{1}u_{1}+\beta_{2}u_{2}+f_{1})-Q_{L}^{T}\hat{f}_{1} \\
	& = \hat{v}_{1}-\alpha_{1}(\alpha_{1} v_{1}+g_{1})-
	\beta_{2}(\alpha_{2}v_{2}+\beta_{2}v_{1}+g_{2}) -Q_{A}^{T}f_{1}-Q_{L}^{T}\hat{f}_{1} \\
	& = (1-\alpha_{1}^{2}-\beta_{2}^{2})\hat{v}_{1} 
	+ \alpha_{2}\beta_{2}\hat{v}_{2} + O(\bar{q}(m,n,p)\epsilon)  .
	\end{align*}
	
	Next, suppose \eqref{3.11} is true for indices up to $i$. For $i+1$, we have
	$$\hat{\alpha}_{i+1}Q_{L}^{T}\hat{u}_{i+1} 
	= Q_{L}^{T}Q_{L}\hat{v}_{i+1}-\hat{\beta}_{i}Q_{L}^{T}\hat{u}_{i}-
	\sum_{j=1}^{i-1}\hat{\xi}_{ji+1}Q_{L}^{T}\hat{u}_{j}-Q_{L}^{T}\hat{f}_{i+1} .$$
	Since $(\hat{\beta}_{i}Q_{L}^{T}\hat{u}_{i}-
	\sum_{j=1}^{i-1}\hat{\xi}_{ji+1}Q_{L}^{T}\hat{u}_{j}) \in span\{\hat{v}_{1}, \dots, \hat{v}_{i+1}\}+ O(\bar{q}(m,n,p)\epsilon)$, we only need to prove $Q_{L}^{T}Q_{L}\hat{v}_{i+1}\in span\{\hat{v}_{1}, \dots, \hat{v}_{i+2}\}+ O(\bar{q}(m,n,p)\epsilon)$. Notice that
	\begin{align*}
	Q_{L}^{T}Q_{L}\hat{v}_{i+1} 
	& = (I_{n}-Q_{A}^{T}Q_{A})\hat{v}_{i+1} \\
	& = \hat{v}_{i+1}+(-1)^{i+1}Q_{A}^{T}(\alpha_{i+1}u_{i+1}+\beta_{i+1}u_{i+2}+\sum_{j=1}^{i}\xi_{ji+1}u_{j}+f_{i+1}) \\
	& = \hat{v}_{i+1}+(-1)^{i+1}(\alpha_{i+1}Q_{A}^{T}u_{i+1}+\beta_{i+1}Q_{A}^{T}u_{i+2}+
	\sum_{j=1}^{i}\xi_{ji+1}Q_{A}^{T}u_{j}) + (-1)^{i+1}Q_{A}^{T}f_{i+1} .
	\end{align*}
	From \eqref{3.9}, we have
	$$(\alpha_{i+1}Q_{A}^{T}u_{i+1}+\beta_{i+1}Q_{A}^{T}u_{i+2}+
	\sum_{j=1}^{i}\xi_{ji+1}Q_{A}^{T}u_{j}) \in span\{\hat{v}_{1}, \dots, \hat{v}_{i+2}\}+ O(\bar{q}(m,n,p)\epsilon) ,$$
	which completes the proof of the induction step.
	
	By mathematical induction principle, \eqref{3.11} holds for all $i = 1, 2, \dots$.
\end{proof}

\begin{proof}[Proof of Lemma \ref{Lem3.2}]
	By \eqref{3.8} and \eqref{3.9}, the process of computing $U_{k+1}$ and $V_{k}$ can be treated as the Lanczos bidiagonalization of $Q_{A}$ with the semiorthogonalization strategy. Since the $k$-step Lanczos bidiagonalization process is equivalent to the $(2k+1)$-step symmetric Lanczos process \cite[\S 7.6.1]{Bjorck1996}, the bounds of $C_{k}$ and $D_{k}$ can be concluded from the property of the symmetric Lanczos process with the semiorthogonalization strategy; see \cite[Lemma 4]{Simon1984b} and its proof.
	
	Now we give the bound of $\widehat{C}_{k}$. At the $(i-1)$-th step, from \eqref{3.10}, we can write the reorthogonalization step of $\hat{u}_{i}$ as 
	\begin{align}
	& \hat{\alpha}_{i}^{'}\hat{u}_{i}^{'} = Q_{L}\hat{v}_{i}-\hat{\beta}_{i-1}\hat{u}_{i-1}-\hat{f}_{i}^{'},\label{A1} \\
	& \hat{\alpha}_{i}\hat{u}_{i} = \hat{\alpha}_{i}^{'}\hat{u}_{i}^{'}-
	\sum_{j=1}^{i-2}\hat{\xi}_{ji}\hat{u}_{j}-\hat{f}_{i}^{''} , \label{A2}
 	\end{align}
	where $\|\hat{f}_{i}^{'}\|, \|\hat{f}_{i}^{''}\| = O(q_{3}(p,n)\epsilon)$. Thus, for $l=1,\dots, i-2$, we have
	$$\hat{\alpha}_{i}^{'}\hat{u}_{l}^{T}\hat{u}_{i}^{'}	
	= \hat{u}_{l}^{T}Q_{L}\hat{v}_{i}-\hat{\beta}_{i-1}\hat{u}_{l}^{T}\hat{u}_{i-1}-
	\hat{u}_{l}^{T}\hat{f}_{i}^{'} .$$
	From \eqref{3.11} and its proof, we know that
	$$ Q_{L}^{T}\hat{u}_{l}=\sum_{j=1}^{l+1}\lambda_{j}\hat{v}_{j} + O(\bar{q}(m,n,p)\epsilon) $$
	with modest constants $\lambda_{j}$ for $j=1,\dots, l+1$. Notice that $\hat{u}_{l}^{T}\hat{u}_{i-1}, \hat{v}_{j}^{T}\hat{v}_{i}\leq\sqrt{\epsilon/(2k+1)}$ for $l=1,\dots, i-2$ and $j=1,\dots, l+1$. We can get
	$$ \hat{\alpha}_{i}^{'}\hat{u}_{l}^{T}\hat{u}_{i}^{'}= \sum_{j=1}^{l+1}\lambda_{j}\hat{v}_{j}^{T}\hat{v}_{i}-
	\hat{\beta}_{i-1}\hat{u}_{l}^{T}\hat{u}_{i-1} + O(\bar{q}(m,n,p)\epsilon) =O(\sqrt{\epsilon}) .$$
	
	Then we prove $M = \max_{1\leq j \leq i-1}|\hat{\xi}_{ji}|=O(\sqrt{\epsilon})$.
	From \eqref{A1}, after being premultiplied by $\hat{u}_{l}^{T}$ and some rearrangement, we obtain
	$$\hat{\xi}_{li} = \hat{\alpha}_{i}^{'}\hat{u}_{l}^{T}\hat{u}_{i}^{'}-\hat{\alpha}_{i}\hat{u}_{l}^{T}\hat{u}_{i}-\sum_{j=1,j\neq l}^{i-2}\hat{\xi}_{ji}\hat{u}_{l}^{T}\hat{u}_{j}-\hat{u}_{l}^{T}\hat{f}_{i}^{''} .$$
	Notice that $\hat{u}_{l}^{T}\hat{u}_{i}=O(\sqrt{\epsilon})$ and we have proved
	$\hat{\alpha}_{i}^{'}\hat{u}_{l}^{T}\hat{u}_{i}^{'}=O(\sqrt{\epsilon})$ for $l=1,\dots, i-2$. We can get
	\begin{equation*}
	|\hat{\xi}_{li}| \leq O(\sqrt{\epsilon})+O(\sqrt{\epsilon})+iM\sqrt{\epsilon}+O(\bar{q}(m,n,p)\epsilon) .
	\end{equation*}
	Now the right-hand side does not depend on $l$ anymore, and we finally obtain by taking the maximum on the left side
	\begin{equation*}
	(1-i\sqrt{\epsilon})M \leq O(\sqrt{\epsilon}) + O(\bar{q}(m,n,p)\epsilon) .
	\end{equation*}
	Therefore, we have $M = O(\sqrt{\epsilon})$.
\end{proof}


%
%



%
%

\end{document}